\documentclass[notitlepage,11pt,reqno]{amsart}
\usepackage[foot]{amsaddr}
\usepackage{amssymb,nicefrac,bm,upgreek,mathtools,verbatim}
\usepackage[final]{hyperref}
\usepackage[mathscr]{eucal}
\usepackage{dsfont}
\usepackage[normalem]{ulem}
\usepackage{amsopn}

\usepackage[margin=1in]{geometry}
\allowdisplaybreaks
\newcommand{\stkout}[1]{\ifmmode\text{\sout{\ensuremath{#1}}}\else\sout{#1}\fi}

\newtheorem{theorem}{Theorem}[section]

\theoremstyle{definition}

\newtheorem{example}{Example}[section]

\theoremstyle{remark}
\newtheorem{remark}{Remark}[section]
\numberwithin{theorem}{section}
\numberwithin{equation}{section}

\hypersetup{
  colorlinks=true,
  citecolor=mblue,
  linkcolor=mblue,
  urlcolor = blue,
  anchorcolor = blue,
  frenchlinks=false,
  pdfborder={0 0 0},
  naturalnames=false,
  hypertexnames=false,
  breaklinks}
\usepackage[capitalize,nameinlink]{cleveref}
\usepackage[abbrev,msc-links,nobysame,citation-order]{amsrefs} 

\crefname{section}{Section}{Sections}
\crefname{subsection}{Section}{Sections}
\crefname{condition}{Condition}{Conditions}
\crefname{hypothesis}{Hypothesis}{Conditions}
\crefname{assumption}{Assumption}{Assumptions}
\crefname{lemma}{Lemma}{Lemmas}
\crefname{fact}{Fact}{Facts}

\Crefname{figure}{Figure}{Figures}

\crefformat{equation}{\textup{#2(#1)#3}}
\crefrangeformat{equation}{\textup{#3(#1)#4--#5(#2)#6}}
\crefmultiformat{equation}{\textup{#2(#1)#3}}{ and \textup{#2(#1)#3}}
{, \textup{#2(#1)#3}}{, and \textup{#2(#1)#3}}
\crefrangemultiformat{equation}{\textup{#3(#1)#4--#5(#2)#6}}%
{ and \textup{#3(#1)#4--#5(#2)#6}}{, \textup{#3(#1)#4--#5(#2)#6}}%
{, and \textup{#3(#1)#4--#5(#2)#6}}

\Crefformat{equation}{#2Equation~\textup{(#1)}#3}
\Crefrangeformat{equation}{Equations~\textup{#3(#1)#4--#5(#2)#6}}
\Crefmultiformat{equation}{Equations~\textup{#2(#1)#3}}{ and \textup{#2(#1)#3}}
{, \textup{#2(#1)#3}}{, and \textup{#2(#1)#3}}
\Crefrangemultiformat{equation}{Equations~\textup{#3(#1)#4--#5(#2)#6}}%
{ and \textup{#3(#1)#4--#5(#2)#6}}{, \textup{#3(#1)#4--#5(#2)#6}}%
{, and \textup{#3(#1)#4--#5(#2)#6}}

\crefdefaultlabelformat{#2\textup{#1}#3}
\newcommand{\vertiii}[1]{{\left\vert\kern-0.25ex\left\vert\kern-0.25ex\left\vert #1 
    \right\vert\kern-0.25ex\right\vert\kern-0.25ex\right\vert}}

\newcommand{\Uadm}{\mathfrak U}
\newcommand{\Act}{\mathbb{U}}
\newcommand{\Usm}{\mathfrak U_{\mathsf{sm}}}

\newcommand{\Um}{\mathfrak U_{\mathsf{m}}}

\newcommand{\pV}{\mathrm{V}} 
\newcommand{\pv}{\mathrm{v}} 

\newcommand{\sB}{{\mathscr{B}}}  
\newcommand{\cC}{{\mathcal{C}}}   
\newcommand{\sE}{{\mathscr{E}}} 
\newcommand{\sF}{{\mathfrak{F}}}   
\newcommand{\cJ}{{\mathcal{J}}}  
  
\newcommand{\sL}{{\mathscr{L}}}  %
\newcommand{\Lp}{{L}}            

\newcommand{\Lyap}{{\mathcal{V}}}  

\newcommand{\RR}{\mathds{R}}
\newcommand{\NN}{\mathds{N}}

\newcommand{\Rd}{{\mathds{R}^{d}}}
\DeclareMathOperator{\Exp}{\mathbb{E}}

\newcommand{\D}{\mathrm{d}}

\newcommand{\Ind}{\mathds{1}}   
\newcommand{\cD}{\mathcal{D}} 
 
\newcommand{\Sob}{{\mathscr W}}    
\newcommand{\Sobl}{{\mathscr W}_{\text{loc}}} 

\newcommand{\df}{:=}
\newcommand{\transp}{^{\mathsf{T}}}
\DeclareMathOperator*{\osc}{osc}

\DeclareMathOperator*{\trace}{Tr}

\newcommand{\sorder}{{\mathfrak{o}}}
\newcommand{\grad}{\nabla}
\newcommand{\uuptau}{{\Breve\uptau}}

\newcommand{\abs}[1]{\lvert#1\rvert}
\newcommand{\norm}[1]{\lVert#1\rVert}

\usepackage{color}
\definecolor{dmagenta}{rgb}{.4,.1,.5}
\definecolor{dblue}{rgb}{.0,.0,.5}
\definecolor{mblue}{rgb}{.0,.0,.7}
\definecolor{ddblue}{rgb}{.0,.0,.4}
\definecolor{dred}{rgb}{.7,.0,.0}
\definecolor{dgreen}{rgb}{.0,.5,.0}
\definecolor{Eeom}{rgb}{.0,.0,.5}

\begin{document}
\title[Robustness of Stochastic Optimal Control for Controlled Diffusion Processes]
{Robustness of Stochastic Optimal Control to Approximate Diffusion Models under Several Cost Evaluation Criteria}

\author[Somnath Pradhan]{Somnath Pradhan$^\dag$}
\address{$^\dag$Department of Mathematics and Statistics,
Queen's University, Kingston, ON, Canada}
\email{sp165@queensu.ca}

\author[Serdar Y\"{u}ksel]{Serdar Y\"{u}ksel$^{\ddag}$}
\address{$^\ddag$Department of Mathematics and Statistics,
Queen's University, Kingston, ON, Canada}
\email{yuksel@queensu.ca}

\begin{abstract}
In control theory, typically a nominal model is assumed based on which an optimal control is designed and then applied to an actual (true) system. This gives rise to the problem of performance loss due to the mismatch between the true model and the assumed model. A robustness problem in this context is to show that the error due to the mismatch between a true model and an assumed model decreases to zero as the assumed model approaches the true model. We study this problem when the state dynamics of the system are governed by controlled diffusion processes. In particular, we will discuss continuity and robustness properties of finite horizon and infinite-horizon $\alpha$-discounted/ergodic optimal control problems for a general class of non-degenerate controlled diffusion processes, as well as for optimal control up to an exit time. Under a general set of assumptions and a convergence criterion on the models, we first establish that the optimal value of the approximate model converges to the optimal value of the true model. We then establish that the error due to mismatch that occurs by application of a control policy, designed for an incorrectly estimated model, to a true model decreases to zero as the incorrect model approaches the true model. We will see that, compared to related results in the discrete-time setup, the continuous-time theory will let us utilize the strong regularity properties of solutions to optimality (HJB) equations, via the theory of uniformly elliptic PDEs, to arrive at strong continuity and robustness properties.
\end{abstract}


\maketitle


\section{Introduction}
In stochastic control applications, typically only an ideal model is assumed, or learned from available incomplete data, based on which an optimal control is designed and then applied to the actual system. This gives rise to the problem of performance loss due to the mismatch between the actual system and the assumed system. A robustness problem in this context is to show that the error due to mismatch decreases to zero as the assumed system approaches the actual system. With this motivation, in this article, our goal is to study the continuity and robustness properties of finite horizon and infinite horizon discounted/ergodic cost problems for a large class of multidimensional controlled diffusions. We note that the problems of existence, uniqueness and verification of optimality of stationary Markov
policies have been studied extensively in literature see e.g., \cite{Bor-book}, \cite{HP09-book} (finite horizon) \cite{BS86}, \cite{BB96} (discounted cost)  \cite{AA12}, \cite{AA13}, \cite{BG88I}, \cite{BG90b} (ergodic cost) and references therein. For a book-length exposition of this topic see e.g., \cite{ABG-book}. 

In more explicit terms, here is the problem that we will study. For a precise statement please see Section \ref{contRobSec}. Suppose that our true model is represented as $(X, c)$ (see, e.g., \cref{E1.1}), where $X$ is the true system model (representing the system evolution model via the drift and diffusion terms) and $c$ is the associated running cost function, and let $(X_n, c_n)$ (see, e.g., \cref{ASE1.1}) be a sequence of approximating models $X_n$ with associated running cost functions $c_n$, such that as $n\to \infty$ approximating models $X_n$ converge to the true model $X$ in some sense to be precisely stated. Suppose that for each choice of control policy $U$ the associated total cost in true/approximating models are $\cJ(c,U),$  $\cJ_n(c_n,U)$ respectively. The objective of the controller is to minimize the total cost over all possible admissible policies. If the optimal control policies of the true/approximating models are $v^*$, $v^{*n}$ respectively, the performance loss due to mismatch is given by $|\cJ(c,v^*) - \cJ(c,v^{*n})|$. Thus the robustness problem in this context is to show that $|\cJ(c,v^*) - \cJ(c,v^{*n})| \to 0$ as $n\to \infty$\,. See Section \ref{contRobSec}. In this sense, our paper can be viewed as a continuous-time counterpart of the setting studied in \cite{KY-20}, \cite{KRY-20}.

This problem is of major practical importance and, accordingly, there have been many studies. Most of the existing works in this direction are concerned with the discrete time Markov decision process, see for instance \cite{KY-20}, \cite{KRY-20}, \cite{BJP02}, \cite{KV16}, \cite{NG05} \cite{SX15}, and references therein.

We should note that the term {\it robustness} has various interpretations, contexts and solution methods. A common approach to robustness in the literature has been to design controllers that work sufficiently well for all possible uncertain systems under some structured constraints, such as $H_\infty$ norm bounded perturbations (see \cite{basbern}, \cite{zhou1996robust}). For such problems, the design for robust controllers has often been developed through a game theoretic formulation where the minimizer is the controller and the maximizer is the uncertainty. In \cite{DJP00}, \cite{jacobson1973optimal} the authors established the connections of this formulation to risk sensitive control. Using Legendre-type transforms, relative entropy constraints came in to the literature to probabilistically model the uncertainties, see e.g. \cite[Eqn. (4)]{dai1996connections} or \cite[Eqns. (2)-(3)]{DJP00}. Here, one selects a nominal system which satisfies a relative entropy bound between the actual measure and the nominal measure, solves a risk sensitive optimal control problem, and this solution value provides an upper bound for the original system performance. Therefore, a common approach in robust stochastic control has been to consider all models which satisfy certain bounds in terms of relative entropy pseudo-distance (or Kullback-Leibler divergence); see e.g. \cite{DJP00,dai1996connections,dupuis2000kernel,boel2002robustness} among others.  In order to quantify the uncertainty in the system models, other than the relative entropy pseudo-distance, various other metrics/criterion have also been used in the literature. In \cite{tzortzis2015dynamic}, for discrete time controlled models, the authors have studied a min-max formulation for robust control where the one-stage transition kernel belongs to a ball under the total variation metric for each state action pair. For distributionally robust stochastic optimization problems, it is assumed that the underlying probability measure of the system lies within an ambiguity set and a worst case single-stage optimization is made considering the probability measures in the ambiguity set. To construct ambiguity sets, \cite{blanchet2016}, \cite{esfahani2015} use the Wasserstein metric, \cite{erdogan2005} uses the Prokhorov metric which metrizes the weak topology, \cite{sun2015} uses the total variation distance, and \cite{lam2016} works with relative entropy. For fully observed finite state-action space models with uncertain transition probabilities, the authors in \cite{iyengar2005robust}, \cite{nilim2005robust} have studied robust dynamic programming approaches through a min-max formulation. Similar work with model uncertainty includes \cite{oksendal2014forward}, \cite{benavoli2011robust}, \cite{xu_mannor}. In the economics literature related work has been done in \cite{hansen2001robust}, \cite{gossner2008entropy}.   
  
The robustness formulation we study has been considered in \cite{KY-20}, \cite{KRY-20} for discrete-time models, where the authors studied both continuity of value functions as transition kernel models converge, as well as the robustness problem where an optimal control designed for an incorrect approximate model is applied to a true model and the mismatch term is studied. The solution approach is fundamentally different in the continuous-time analysis we present in this paper. In a related study \cite{Dean18}, the author studied the optimal control of systems with unknown dynamics for a linear quadratic regulator setup and proposes an algorithm to learn the system from observed data with quantitative convergence bounds. The author in \cite[Theorem 5.1]{Lan81} has considered fully observed discrete time controlled models and established continuity results for approximate models and gives a set convergence result for sets of optimal control actions, this set convergence result is inconclusive for robustness without further assumptions on the true system model (for more details see \cite{KY-20}). For fully observed MDPs, \cite{muller1997does} studied continuity of the value function under a general metric defined as the integral probability metric, which captures both the total variation metric or the Kantorovich metric with different setups (which is not weaker than the metrics leading to weak convergence). A recent study on game problems along a similar theme is presented in \cite{subramanian2021robustness}.

For control problems of MDPs with standard Borel spaces, the approximation methods through quantization, which lead to finite models, can be viewed as approximations of transition kernels, but this interpretation requires caution: indeed, \cite{SaYuLi17, arruda2012, arruda2013}, among many others, study approximation methods for MDPs where the convergence of approximate models is satisfied in a particularly constructed fashion. Reference \cite{SaYuLi17} presents a construction for the approximate models through quantizing the actual model with continuous spaces (leading to a finite space model), which allows for continuity and robustness results with only a weak continuity assumption on the true transition kernel which, in turn,  leads to the weak convergence of the approximate models. For both fully observed and partially observed models, a detailed analysis of approximation methods for continuous state and action spaces can be found in \cite{SaLiYuSpringer} .

The literature on robustness of stochastic optimal control for continuous time system seems to be rather limited; see e.g., \cite{GL99}, \cite{LJE15} \cite{hansen2001robust}\,. In \cite{GL99} the authors have considered the problem of controlling a system whose dynamics are given by a stochastic differential equation (SDE) whose coefficients are known only up to a certain degree of accuracy. For the finite horizon reward maximization problem, using the technique of contractive operators, \cite{GL99} has obtained upper bounds of performance loss due to mismatch (or, ``robustness index") and has shown by an example that the robustness index may behave abnormally even if we have the convergence of the value functions. The associated discounted payoff maximization problem has been studied in \cite{LJE15}, where using a Lyapunov type stability assumption the authors have studied the robustness problem via a game theoretic formulation. For controlled diffusion models, the authors in \cite{hansen2001robust} described the links between the max-min expected utility theory and the applications of robust-control theory, in analogy with some of the papers on discrete-time noted above adopting a min-max formulation. Along a further direction, for controlled diffusions, via the Maximum Principle technique, \cite{PDPB02a}, \cite{PDPB02b}, \cite{PDPB02c} have established the robustness of optimal controls for the finite horizon payoff criterion.

In a recent comprehensive work \cite{RZ21}, the authors have studied the robustness of feedback relaxed controls for a continuous time stochastic exit time problem. Under sufficient smoothness assumptions on the coefficients (i.e, uniform Lipschitz continuity on the diffusion coefficients and uniform H\"older continuity on the discount factor and payoff function on a fixed bounded domain) they have established that a regularized control problem admits a H\"older continuous optimal feedback control and also they have shown that both the value function and the feedback control of the regularized control problem are Lipschitz stable with respect to parameter perturbations when the action space is finite. It is known that the optimal control obtained form the HJB equation (i.e. the argmin function) in general is unstable with respect to perturbations of coefficients; in practice, this would result in numerical instability of learning algorithms (as noted in \cite{RZ21}).

Stability/continuity of solutions of PDEs with respect coefficient perturbations is a significant mathematical and practical question in PDE theory (see e.g. \cite{WLS01}, \cite{SI72}). The continuity results established in this paper (see Theorems~\ref{TC1.3}, \ref{ErgoContnuity}, \ref{TErgoOptCont}) will provide sufficient conditions which ensure stability of solutions of semilinear elliptic PDEs (HJB equations) in the whole space $\Rd$.

Our robustness results also will be useful to the study of the robust optimal investment problems for local volatility models, e.g. given in \cite[Remark~2.1]{AS08} (also, see \cite{KT12}, \cite{BDD20})\,.          

When the system noise is not given by a Wiener process, but it is given by a general wide bandwidth noise (or, a more general discontinuous martingales \cite{LRT00}), the controlled process becomes a non-Markovian process even under stationary Markov policies. The general method of studying optimal control problem for such a system is to find suitable Markovian processes which approximate the non-Markovian process (see, \cite{K90}, \cite{KR87}, \cite{KR87a}, \cite{KR88}). For wide bandwidth noise driven controlled systems \cite{K90}, \cite{KR87}, \cite{KR87a}, \cite{KR88}, diffusion approximation techniques were used to study stochastic optimal control problems. The results described in this paper are complementary to the above mentioned works on the diffusion approximation of wide bandwidth noise driven systems. 

{\bf Contributions and main results.} In the present paper, our aim is to study the continuity and robustness properties for a general class of controlled diffusion processes in $\Rd$ for both infinite horizon discounted/ ergodic costs, where the action space is a (general) compact metric space. As in \cite{KY-20}, \cite{KRY-20}, in order to establish our desired robustness results we will use the continuity result as an intermediate step. For the discounted cost case, we will establish our results following a direct approach (under a relatively weaker set of assumptions on the diffusion coefficients, i.e., locally Lipschitz continuous coefficients). Using the results on existence and uniqueness of solutions of the associated discounted Hamilton Jacobi Bellman (HJB) equation and the complete characterization of (discounted) optimal policies in the space of stationary Markov policies (see \cite[Theorem~3.5.6]{ABG-book}), we first establish the continuity of value functions. Then utilizing this continuity of value functions, we derive a robustness result. The analysis of ergodic cost (or long-run expected average cost) is somewhat more involved. To the best of our knowledge there is no work on continuity and robustness properties of optimal controls for the ergodic cost criterion in the existing literature (for the discrete-time setup, see \cite{KRY-20}). We have studied these ergodic cost problems under two sets of assumptions: In the first case, we assume that our running cost function satisfies a near-monotone type structural assumption (see, eq. \cref{ENearmonot}, Assumption~(A6)), and in the second case we assume Lyapunov type stability assumptions on the dynamics of the system (see Assumption~(A7))\,. 

One of the major issues in analyzing the robustness of ergodic optimal controls under the near-monotone hypothesis is the non-uniqueness/restricted uniqueness of solutions of the associated HJB equation (see, \cite[Example~3.8.3]{ABG-book}, \cite{AA13}). It is shown in \cite[Example~3.8.3]{ABG-book} that the ergodic HJB equation may admit uncountably many solutions. Considering this, in \cite[Theorem~1.1]{AA13} the author has established the uniqueness of compatible solution pairs (see \cite[Definition~1.1]{AA13}). Exploiting this uniqueness result, under a suitable tightness assumption (on a certain set of invariant measures) we will establish the desired robustness result. Under the Lyapunov type stability assumption it is known that the ergodic HJB equation admits a unique solution in a certain class of functions, also the complete characterization of ergodic optimal control is known (see \cite[Theorem~3.7.11]{ABG-book} and \cite[Theorem~3.7.12]{ABG-book})\,. Utilizing this characterization of optimal controls, we derive the robustness properties of ergodic optimal controls under a Lyapunov stability assumption. 

We also emphasize the duality between the PDE approach vs. a probabilistic flow approach to study robustness. The PDE approach presents a very general and conclusive, yet concise and unified, approach for several cost criteria (notably, a probabilistic approach via Dynkin's lemma would require separate arguments for discounted infinite-horizon and average cost infinite-horizon criteria) and such a unified approach had not been considered earlier, to our knowledge.

Thus, the main results of this article can be roughly described as follows. 
\begin{itemize}
\item[•] {\it For discounted cost criterion:} We establish continuity of value functions and provide sufficient conditions which ensure robustness/stability of optimal controls designed under model uncertainties.
\item[•] {\it For ergodic cost criterion:} Under two different sets of assumptions ((i) where the running cost is near-monotone or (ii) where a Lyapunov stability condition holds) we establish the continuity of value functions and exploiting the continuity results we derive the robustness/stability of ergodic optimal controls designed for approximate models applied to actual systems.
\item[•] {\it For finite horizon cost criterion:} Under uniform boundedness assumptions on the drift term and diffusion matrices (of the true and approximating models), we establish continuity of value functions. Then exploiting the continuity result we prove the robustness/stability of optimal controls designed under model uncertainties.
\item[•] {\it For cost up to an exit time:} Similar to the above criteria, under a mild set of assumptions we first establish the continuity of value functions and then using the continuity results we establish the robustness/stability of optimal controls designed under model uncertainties. 
\end{itemize}          

We will see that compared with the discrete-time counterpart of this problem studied in \cite{KY-20} (discounted cost) and \cite{KRY-20} (average cost), where value iteration methods were crucially used, in our analysis here we will develop rather direct arguments, with strong implications, utilizing regularity properties of value functions: In the discrete-time setup, these properties need to be established via tedious arguments whereas the continuous-time theory allows for the use of regularity properties of solutions to PDEs. Nonetheless, we will see that {\it continuous convergence in control actions} of models and cost functions is a unifying condition for continuity and robustness properties in both the discrete-time setup studied in \cite{KY-20} (discounted cost) and \cite{KRY-20} (average cost) and our current paper. Compared to \cite{RZ21}, in addition to the infinite horizon criteria we study, we note that the perturbations we consider do not involve only coefficient/parameter variations, i.e., we consider functional perturbations, and the action space we consider is uncountable, though we do not establish the Lipschitz property of control policies, unlike \cite{RZ21}.

The rest of the paper is organized as follows. Section~\ref{PD} introduces the the problem setup and summarizes the notation. Section~\ref{discCostSec} is devoted to the analysis of robustness of optimal controls for discounted cost criterion. In Section~\ref{secErgodicCost} we provide the analysis of robustness of ergodic optimal control under two different sets of hypotheses (i) near-monotonicity (ii) Lyapunov stability. For the finite horizon cost criterion the robustness problem is analyzed in Section~\ref{Finitecost}. The robustness problem for optimal controls up to an exit time is considered in Section~\ref{exitTimeSection}. 
\section{Description of the problem}\label{PD} Let $\Act$ be a compact metric space and $\pV=\mathscr{P}(\Act)$ be the space of probability measures on  $\Act$ with topology of weak convergence. Let $$b : \Rd \times \Act \to  \Rd, $$ $$ \sigma : \Rd \to \RR^{d \times d},\, \sigma = [\sigma_{ij}(\cdot)]_{1\leq i,j\leq d},$$ be given functions. We consider a stochastic optimal control problem whose state is evolving according to a controlled diffusion process given by the solution of the following stochastic differential equation (SDE)
\begin{equation}\label{E1.1}
\D X_t \,=\, b(X_t,U_t) \D t + \upsigma(X_t) \D W_t\,,
\quad X_0=x\in\Rd.
\end{equation}
Where 
\begin{itemize}
\item
$W$ is a $d$-dimensional standard Wiener process, defined on a complete probability space $(\Omega, \sF, \mathbb{P})$.
\item 
 We extend the drift term $b : \Rd \times \pV \to  \Rd$ as follows:
\begin{equation*}
b (x,\mathrm{v}) = \int_{\Act} b(x,\zeta)\mathrm{v}(\D \zeta), 
\end{equation*}
for $\mathrm{v}\in\pV$.
\item
$U$ is a $\pV$ valued process satisfying the following non-anticipativity condition: for $s<t\,,$ $W_t - W_s$ is independent of
$$\sF_s := \,\,\mbox{the completion of}\,\,\, \sigma(X_0, U_r, W_r : r\leq s)\,\,\,\mbox{relative to} \,\, (\sF, \mathbb{P})\,.$$  
\end{itemize}
The process $U$ is called an \emph{admissible} control, and the set of all admissible controls is denoted by $\Uadm$ (see, \cite{BG90}).

To ensure existence and uniqueness of strong solutions of \cref{E1.1}, we impose the following assumptions on the drift $b$ and the diffusion matrix $\upsigma$\,. 
\begin{itemize}
\item[\hypertarget{A1}{{(A1)}}]
\emph{Local Lipschitz continuity:\/}
The function
$\upsigma\,=\,\bigl[\upsigma^{ij}\bigr]\colon\RR^{d}\to\RR^{d\times d}$,
$b\colon\Rd\times\Act\to\Rd$ are locally Lipschitz continuous in $x$ (uniformly with respect to the control action for $b$). In particular, for some constant $C_{R}>0$
depending on $R>0$, we have
\begin{equation*}
\abs{b(x,\zeta) - b(y, \zeta)}^2 + \norm{\upsigma(x) - \upsigma(y)}^2 \,\le\, C_{R}\,\abs{x-y}^2
\end{equation*}
for all $x,y\in \sB_R$ and $\zeta\in\Act$, where $\norm{\upsigma}\df\sqrt{\trace(\upsigma\upsigma\transp)}$\,. Also, we are assuming that $b$ is jointly continuous in $(x,\zeta)$.

\medskip
\item[\hypertarget{A2}{{(A2)}}]
\emph{Affine growth condition:\/}
$b$ and $\upsigma$ satisfy a global growth condition of the form
\begin{equation*}
\sup_{\zeta\in\Act}\, \langle b(x, \zeta),x\rangle^{+} + \norm{\upsigma(x)}^{2} \,\le\,C_0 \bigl(1 + \abs{x}^{2}\bigr) \qquad \forall\, x\in\RR^{d},
\end{equation*}
for some constant $C_0>0$.

\medskip
\item[\hypertarget{A3}{{(A3)}}]
\emph{Nondegeneracy:\/}
For each $R>0$, it holds that
\begin{equation*}
\sum_{i,j=1}^{d} a^{ij}(x)z_{i}z_{j}
\,\ge\,C^{-1}_{R} \abs{z}^{2} \qquad\forall\, x\in \sB_{R}\,,
\end{equation*}
and for all $z=(z_{1},\dotsc,z_{d})\transp\in\RR^{d}$,
where $a\df \frac{1}{2}\upsigma \upsigma\transp$.
\end{itemize}

By a Markov control we mean an admissible control of the form $U_t = v(t,X_t)$ for some Borel measurable function $v:\RR_+\times\Rd\to\pV$. The space of all Markov controls is denoted by $\Um$\,.
If the function $v$ is independent of $t$, then $U$ or by an abuse of notation $v$ itself is called a stationary Markov control. The set of all stationary Markov controls is denoted by $\Usm$. From \cite[Section~2.4]{ABG-book}, we have that the set $\Usm$ is metrizable with compact metric under the following topology: A sequence $v_n\to v$ in $\Usm$ if and only if
\begin{equation*}
\lim_{n\to\infty}\int_{\Rd}f(x)\int_{\Act}g(x,u)v_{n}(x)(\D u)\D x = \int_{\Rd}f(x)\int_{\Act}g(x,u)v(x)(\D u)\D x
\end{equation*}
for all $f\in L^1(\Rd)\cap L^2(\Rd)$ and $g\in \cC_b(\Rd\times \Act)$ (for more details, see \cite[Lemma~2.4.1]{ABG-book})\,. It is well known that under the hypotheses \hyperlink{A1}{{(A1)}}--\hyperlink{A3}{{(A3)}}, for any admissible control \cref{E1.1} has a unique strong solution \cite[Theorem~2.2.4]{ABG-book}, and under any stationary Markov strategy \cref{E1.1} has unique strong solution which is a strong Feller (therefore strong Markov) process \cite[Theorem~2.2.12]{ABG-book}.

\subsection{Cost Criteria}

Let $c\colon\Rd\times\Act \to \RR_+$ be the \emph{running cost} function. We assume that 
\begin{itemize}
\item[\hypertarget{A4}{{(A4)}}]
The \emph{running cost} $c$ is bounded (i.e., $\|c\|_{\infty} \leq M$ for some positive constant $M$), jointly continuous in $(x, \zeta)$ and locally Lipschitz continuous in its first argument uniformly with respect to $\zeta\in\Act$.
\end{itemize}
This condition (A4) can also be relaxed to (A4\'), to be presented further below, where the local Lipschitz property is eliminated.

 We extend $c\colon\Rd\times\pV \to\RR_+$ as follows: for $\pv \in \pV$
\begin{equation*}
c(x,\pv) := \int_{\Act}c(x,\zeta)\pv(\D\zeta)\,.
\end{equation*}

In this article, we consider the problem of minimizing finite horizon, discounted, ergodic and control up to an exit time cost criteria:\\

{\bf Discounted cost criterion.}  For $U \in\Uadm$, the associated \emph{$\alpha$-discounted cost} is given by
\begin{equation}\label{EDiscost}
\cJ_{\alpha}^{U}(x, c) \,\df\, \Exp_x^{U} \left[\int_0^{\infty} e^{-\alpha s} c(X_s, U_s) \D s\right],\quad x\in\Rd\,,
\end{equation} where $\alpha > 0$ is the discount factor
and $X(\cdot)$ is the solution of \cref{E1.1} corresponding to $U\in\Uadm$ and $\Exp_x^{U}$ is the expectation with respect to the law of the process $X(\cdot)$ with initial condition $x$. The controller tries to minimize \cref{EDiscost} over his/her admissible policies $\Uadm$\,. Thus, a policy $U^{*}\in \Uadm$ is said to be optimal if for all $x\in \Rd$ 
\begin{equation}\label{OPDcost}
\cJ_{\alpha}^{U^*}(x, c) = \inf_{U\in \Uadm}\cJ_{\alpha}^{U}(x, c) \,\,\, (\,=:\, \,\, V_{\alpha}(x))\,,
\end{equation} where $V_{\alpha}(x)$ is called the optimal value.\\

{\bf Ergodic cost criterion.} For each $U\in\Uadm$ the associated ergodic cost is defined as
\begin{equation}\label{ECost1}
\sE_{x}(c, U) = \limsup_{T\to \infty}\frac{1}{T}\Exp_x^{U}\left[\int_0^{T} c(X_s, U_s) \D{s}\right]\,,
\end{equation} and the optimal value is defined as
\begin{equation}\label{ECost1Opt}
\sE^*(c) \,\df\, \inf_{x\in\Rd}\inf_{U\in \Uadm}\sE_{x}(c, U)\,.
\end{equation}
Then a control $U^*\in \Uadm$ is said to be optimal if we have 
\begin{equation}\label{ECost1Opt1}
\sE_{x}(c, U^*) = \sE^*(c)\quad \text{for all}\,\,\, x\in \Rd\,.
\end{equation}\\

{\bf Finite horizon cost.}  For $U\in \Uadm$, the associated \emph{finite horizon cost} is given by
\begin{equation}\label{FiniteCost1}
\cJ_{T}^U(x, c) = \Exp_x^{U}\left[\int_0^{T} c(X_s, U_s) \D{s} + H(X_T)\right]\,,
\end{equation} where $H(\cdot)$ is the terminal cost. The optimal value is defined as
\begin{equation}\label{FiniteCost1Opt}
\cJ_{T}^*(x, c) \,\df\, \inf_{U\in \Uadm}\cJ_{T}^U(x, c)\,.
\end{equation}
Thus, a policy $U^*\in \Uadm$ is said to be (finite horizon) optimal if we have 
\begin{equation}\label{FiniteCost1Opt1}
\cJ_{T}^{U^*}(x, c) = \cJ_{T}^*(x, c)\quad \text{for all}\,\,\, x\in \Rd\,.
\end{equation}\\

{\bf Control up to an exit time.} This criterion will be presented in Section \ref{exitTimeSection}. Our analysis for this criterion will be immediate given the study involving the above criteria. \\

We define a family of operators $\sL_{\zeta}$ mapping
$\cC^2(\Rd)$ to $\cC(\Rd)$ by
\begin{equation}\label{E-cI}
\sL_{\zeta} f(x) \,\df\, \trace\bigl(a(x)\grad^2 f(x)\bigr) + \,b(x,\zeta)\cdot \grad f(x)\,, 
\end{equation}
for $\zeta\in\Act$, \,\, $f\in \cC^2(\Rd)$\,.
For $\pv \in\pV$ we extend $\sL_{\zeta}$ as follows:
\begin{equation}\label{EExI}
\sL_\pv f(x) \,\df\, \int_{\Act} \sL_{\zeta} f(x)\pv(\D \zeta)\,.
\end{equation} For $v \in\Usm$, we define
\begin{equation}\label{Efixstra}
\sL_{v} f(x) \,\df\, \trace(a\grad^2 f(x)) + b(x,v(x))\cdot\grad f(x)\,.
\end{equation}
We are interested in the robustness of optimal controls under these criteria. To this end, we now introduce our approximating models.

\subsection{Approximating Control Diffusion Process:}
Let, $\upsigma_n\,=\,\bigl[\upsigma_n^{ij}\bigr]\colon\RR^{d}\to\RR^{d\times d}$, $b_n\colon\Rd\times\Act\to\Rd$, $c_n\colon\Rd\times\Act\to\Rd$ be sequence of functions satisfying the following assumptions 
\begin{itemize}
\item[\hypertarget{A5}{{(A5)}}]
\begin{itemize}
\item[(i)] as $n\to\infty$ 
\begin{equation}\label{AproxiE1}
\upsigma_n(x)\to \upsigma(x)\quad \text{a.e.}\,\, x\in\Rd\,,
\end{equation}
\item[(ii)]\emph{Continuous convergence in controls}: for any sequence $\zeta_n\to \zeta$
\begin{equation}\label{AproxiE2}
c_n(x,\zeta_n)\to c(x,\zeta)\quad\text{and}\quad b_n(x,\zeta_n)\to b(x,\zeta)\quad \text{a.e.}\,\, x\in\Rd\,.  
\end{equation}
\item[(iii)] for each $n\in\NN$,\, $b_n$ and $\upsigma_n$ satisfy Assumptions (A1) - (A3) and $c_n$ is uniformly bounded ( in particular, $\norm{c_n}_{\infty} \leq M$ where $M$ is a positive constant as in (A4)), jointly continuous in $(x, \zeta)$ and locally Lipschitz continuous in its first argument uniformly with respect to $\zeta\in\Act$.\,. 
\end{itemize}
\end{itemize}
Let for each $n\in\NN$, $X_t^n$ be the solution of the following SDE
\begin{equation}\label{ASE1.1}
\D X_t^n \,=\, b_n(X_t^n,U_t) \D t + \upsigma_n(X_t^n) \D W_t\,,
\quad X_0^n=x\in\Rd.
\end{equation}
Define a family of operators $\sL_{\zeta}^n$ mapping
$\cC^2(\Rd)$ to $\cC(\Rd)$ by
\begin{equation}\label{E-cIn}
\sL_{\zeta}^n f(x) \,\df\, \trace\bigl(a_n(x)\grad^2 f(x)\bigr) + \,b_n(x,\zeta)\cdot \grad f(x)\,, 
\end{equation}
for $\zeta\in\Act$, \,\, $f\in \cC^2(\Rd)$\,. For the approximated model, for each $n\in\NN$ and $U\in\Uadm$ the associated discounted cost is defined as
\begin{equation}\label{EApproDiscost}
\cJ_{\alpha, n}^{U}(x, c_n) \,\df\, \Exp_x^{U} \left[\int_0^{\infty} e^{-\alpha t} c_n(X_s^n, U_s) \D s\right],\quad x\in\Rd\,,
\end{equation} and the optimal value is defined as 
\begin{equation}\label{EApproOptDisc}
V_{\alpha}^n(x) \,\df\, \inf_{U\in\Uadm}\cJ_{\alpha, n}^{U}(x, c_n)
\end{equation}

For each $n\in\NN$ and $U\in\Uadm$ the associated ergodic cost is defined as   
\begin{equation}\label{ECostAprox1}
\sE_{x}^n(c_n, U) = \limsup_{T\to \infty}\frac{1}{T}\Exp_x^{U}\left[\int_0^{T} c_n(X_s^n, U_s) \D{s}\right]\,,
\end{equation} and the optimal value is defined as
\begin{equation}\label{ECost1OptAprox}
\sE^{n*}(c_n) \,\df\, \inf_{x\in\Rd}\inf_{U\in \Uadm}\sE_{x}^n(c_n, U)\,.
\end{equation} 
Similarly, for each $n\in\NN$ and $U\in\Uadm$ the associated finite horizon cost is given by   
\begin{equation}\label{FiniteCost1OptAprox1}
\cJ_{T,n}^U(x,c_n) \,\df\, \Exp_x^{U}\left[\int_0^{T} c_n(X_s^n, U_s) \D{s} + H(X_T^n)\right]\,.
\end{equation}
The optimal value is given by 
\begin{equation}\label{FiniteCost1Opt1Aprox2}
\cJ_{T,n}^*(x, c_n) = \inf_{U\in \Uadm}\cJ_{T,n}^U (x, c_n)\quad \text{for all}\,\,\, x\in \Rd\,,
\end{equation} where state process $X_t^n$ is given by the solution of the SDE \cref{ASE1.1}\,.

\subsection{Continuity and Robustness Problems}\label{contRobSec}

The primary objective of this article will be to address the following problems:
\begin{itemize}
\item \textbf{Continuity:} If the approximated model \cref{ASE1.1} approches the true model \cref{E1.1}, whether this implies
\begin{itemize}
\item[•] for discounted cost: $V_{\alpha}^n \to V_{\alpha} ?$
\item[•] for ergodic cost : $\sE^{n*}(c_n) \to \sE^{*}(c) ?$
\item[•] for finite horizon cost : $\cJ_{T,n}^*(x, c_n) \to \cJ_{T}^*(x, c) ?$
\item[•] for cost up to an exit time: $\hat{\cJ}_{e,n}^* \to \hat{\cJ}_{e}^*?$ \,\,\, (for details, see Section~\ref{exitTimeSection})
\end{itemize} 
\item \textbf{Robustness:} Suppose $v_{n}^*$ is an optimal policy designed over incorrect model \cref{ASE1.1} for finite horizon/ discounted/ergodic/up to an exit time cost problem, does this imply 
\begin{itemize}
\item[•] for discounted cost: $\cJ_{\alpha}^{v_n^*}(x, c) \to V_{\alpha} ?$
\item[•] for ergodic cost: $\sE_x(c, v_n^*) \to \sE^*(c) ? $
\item[•] for finite horizon cost : $\cJ_{T}^{v_n^*}(x, c) \to \cJ_{T}^*(x, c) ?$
\item[•] for cost up to an exit time: $\hat{\cJ}_{e}^{v_n^*} \to \hat{\cJ}_{e}^*?$ \,\,\, (for details, see Section~\ref{exitTimeSection})
\end{itemize}as $n\to \infty$\,. 
\end{itemize}
In this article, under a mild set of assumptions we show that the answers to the above mentioned questions are affirmative. 

\begin{example}\label{ERS11Example}
\begin{itemize}
\item[(i)]
If our noise term is not the (ideal) Brownian, and instead of \cref{E1.1}, the state dynamics of the system is governed the following SDE
\begin{equation}\label{ERS1.1}
\begin{cases}
\D \hat{X}_t^n \,=\, b(\hat{X}_t^n,U_t) \D t + \upsigma(\hat{X}_t^n) \D S_t^{n}\\
\D S_t^{n} \,=\, \hat{b}_n(\hat{X}_t^n) \D t + \hat{\upsigma}_n(\hat{X}_t^n) \D \hat{W}_t\,.
\end{cases}
\end{equation}
Here we are approximating the noise term by a It$\hat{\rm o}$ process $\{S_t^{n}\}$, given by
\begin{equation}\label{ERS1.1A}
\D S_t^{n} \,=\, \hat{b}_n(\hat{X}_t^n) \D t + \hat{\upsigma}_n(\hat{X}_t^n) \D \hat{W}_t\,,
\end{equation} where $\hat{b}_n(\cdot)\to 0$ and $\hat{\upsigma}_n(\cdot)\to I$ as $n\to\infty$. 
\item[(ii)] 
Suppose that \cref{E1.1} is approximated by \cref{ASE1.1} where $b_n$ and $\upsigma_n$ consist of polynomials of appropriate dimensions which converge pointwise to $b$ and $\upsigma$ (which are already assumed to be continuous) where we also have continuous convergence in control variable $\zeta$. 
\item[(iii)]
 Consider a Vasicek interest rate model, given by
\begin{equation*}
\D r_t \,=\, \theta(\mu - r_t)\D t + \sigma\D W_t\,.
\end{equation*} this is a mean reverting process, where $\theta$ is the rate of reversion, $\mu$ is the long term mean and $\sigma$ is the volatility. The wealth process corresponding to this interest model can be described by \cref{E1.1} (see \cite[Remark~2.1]{AS08}, \cite{KT12}, \cite{DJ07})). Since market models are typically incomplete, usually model parameters ($\theta, \mu, \sigma$) are learned from the market data. This gives rise to the problem of robustness of optimal investment. This also applies to several other interest/pricing models as well, \cite{merton1998applications}. 

\item[(iv)] 
In the above examples, $c_n$ can be a regularized version of $c$, e.g. by adding, for $\epsilon_n > 0$, $\epsilon_n \zeta^T \zeta$ where $\mathbb{U} \subset \mathbb{R}^m$, which then would continuously converge (in control) to $c$ as $\epsilon_n \to 0$. 
\end{itemize}
In the cases above, the approximating kernel conditions in (A5) would apply.
\end{example}
\begin{remark}
If we replace $\sigma(x)$ by $\sigma(x,\zeta)$, in the relaxed control framework if $\sigma(\cdot, v(\cdot))$ is Lipschitz continuous for $v\in \Usm$ then $\cref{E1.1}$ admits a unique strong solution. But in general stationary policies $v\in \Usm$ are just measurable functions. Existence of suitable strong solutions in our setting is not known (see, \cite[Remarks~2.3.2]{ABG-book}, \cite{B05Survey})\,. However, under stationary Markov policies one can prove the existence of weak solutions which may not be unique \cite{stroock1997multidimensional}\cite[Remarks~2.3.2]{ABG-book} (note though that uniqueness is established for $d=1,2$ in \cite[p. 192-194]{stroock1997multidimensional} under some conditions). The existence of a suitable strong solution (which is also a strong Markov process) under stationary Markov policies is essential to obtain stochastic representation of solutions of HJB equations (by applying It$\hat{o}$-Krylov formula).      
\end{remark}

\subsection*{Notation:}
\begin{itemize}
\item For any set $A\subset\RR^{d}$, by $\uptau(A)$ we denote \emph{first exit time} of the process $\{X_{t}\}$ from the set $A\subset\RR^{d}$, defined by
\begin{equation*}
\uptau(A) \,\df\, \inf\,\{t>0\,\colon X_{t}\not\in A\}\,.
\end{equation*}
\item $\sB_{r}$ denotes the open ball of radius $r$ in $\RR^{d}$, centered at the origin, and $\sB_{r}^c$ denotes the complement of $\sB_{r}$ in $\Rd$\,.
\item $\uptau_{r}$, $\uuptau_{r}$ denote the first exist time from $\sB_{r}$, $\sB_{r}^c$ respectively, i.e., $\uptau_{r}\df \uptau(\sB_{r})$, and $\uuptau_{r}\df \uptau(\sB^{c}_{r})$.
\item By $\trace S$ we denote the trace of a square matrix $S$.
\item For any domain $\cD\subset\RR^{d}$, the space $\cC^{k}(\cD)$ ($\cC^{\infty}(\cD)$), $k\ge 0$, denotes the class of all real-valued functions on $\cD$ whose partial derivatives up to and including order $k$ (of any order) exist and are continuous.
\item $\cC_{\mathrm{c}}^k(\cD)$ denotes the subset of $\cC^{k}(\cD)$, $0\le k\le \infty$, consisting of functions that have compact support. This denotes the space of test functions.
\item $\cC_{b}(\Rd)$ denotes the class of bounded continuous functions on $\Rd$\,.
\item $\cC^{k}_{0}(\cD)$, denotes the subspace of $\cC^{k}(\cD)$, $0\le k < \infty$, consisting of functions that vanish in $\cD^c$.
\item $\cC^{k,r}(\cD)$, denotes the class of functions whose partial derivatives up to order $k$ are H\"older continuous of order $r$.
\item $\Lp^{p}(\cD)$, $p\in[1,\infty)$, denotes the Banach space
of (equivalence classes of) measurable functions $f$ satisfying
$\int_{\cD} \abs{f(x)}^{p}\,\D{x}<\infty$.
\item $\Sob^{k,p}(\cD)$, $k\ge0$, $p\ge1$ denotes the standard Sobolev space of functions on $\cD$ whose weak derivatives up to order $k$ are in $\Lp^{p}(\cD)$, equipped with its natural norm (see, \cite{Adams})\,.
\item  If $\mathcal{X}(Q)$ is a space of real-valued functions on $Q$, $\mathcal{X}_{\mathrm{loc}}(Q)$ consists of all functions $f$ such that $f\varphi\in\mathcal{X}(Q)$ for every $\varphi\in\cC_{\mathrm{c}}^{\infty}(Q)$. In a similar fashion, we define $\Sobl^{k, p}(\cD)$.
\item  For $\mu > 0$, let $e_{\mu}(x) = e^{-\mu\sqrt{1+\abs{x}^2}}$\,, $x\in\Rd$\,. Then $f\in \Lp^{p,\mu}((0, T)\times \Rd)$ if $fe_{\mu} \in \Lp^{p}((0, T)\times \Rd)$\,. Similarly, $\Sob^{1,2,p,\mu}((0, T)\times \Rd) = \{f\in \Lp^{p,\mu}((0, T)\times \Rd) \mid f, \frac{\partial f}{\partial t}, \frac{\partial f}{\partial x_i}, \frac{\partial^2 f}{\partial x_i \partial x_j}\in \Lp^{p,\mu}((0, T)\times \Rd) \}$ with natural norm (see \cite{BL84-book})
\begin{align*}
\norm{f}_{\Sob^{1,2,p,\mu}} = \norm{\frac{\partial f}{\partial t}}_{\Lp^{p,\mu}((0, T)\times \Rd)} + \norm{f}_{\Lp^{p,\mu}((0, T)\times \Rd)}  + & \sum_{i}\norm{\frac{\partial f}{\partial x_i}}_{\Lp^{p,\mu}((0, T)\times \Rd)} \nonumber\\
&+ \sum_{i,j}\norm{\frac{\partial^2 f}{\partial x_i \partial x_j}}_{\Lp^{p,\mu}((0, T)\times \Rd)}\,.
\end{align*}
\end{itemize}
\section{Analysis of Discounted Cost}\label{discCostSec}
In this section we analyze the robustness of optimal controls for discounted cost criterion. From \cite[Theorem~3.5.6]{ABG-book}, we have the following characterization of the optimal $\alpha$-discounted cost $V_{\alpha}$\,.

\begin{theorem}\label{TD1.1}
Suppose Assumptions (A1)-(A4) hold. Then the optimal discounted cost $V_{\alpha}$ defined in \cref{OPDcost} is the unique solution in $\cC^2(\Rd)\cap\cC_b(\Rd)$ of the HJB equation
\begin{equation}\label{OptDHJB}
\min_{\zeta \in\Act}\left[\sL_{\zeta}V_{\alpha}(x) + c(x, \zeta)\right] = \alpha V_{\alpha}(x) \,,\quad \text{for all\ }\,\, x\in\Rd\,.
\end{equation}
Moreover, $v^*\in \Usm$ is $\alpha$-discounted optimal control if and only if it is a measurable minimizing selector of\cref{OptDHJB}, i.e.,
\begin{equation}\label{OtpHJBSelc}
b(x,v^*(x))\cdot \grad V_{\alpha}(x) + c(x, v^*(x)) = \min_{\zeta\in \Act}\left[ b(x, \zeta)\cdot \grad V_{\alpha}(x) + c(x, \zeta)\right]\quad \text{a.e.}\,\,\, x\in\Rd\,.
\end{equation}
\end{theorem}
\begin{remark}
The assumption that the running cost is Lipschitz continuous in it's first argument uniformly with respect to the second, is used to obtain a $\cC^2(\Rd)$ solution of the HJB equation \cref{OptDHJB}. If we don't have this uniformly Lipschitz assumption, one can still show that the HJB equation admits a solution now in $\Sobl^{2, p}(\Rd)$,\, $p\geq d+1$ and all the conclusions of the Theorem~\ref{TD1.1} still hold\,. To see this: in view of \cite[Theorem~9.15]{GilTru} and the Schauder fixed point theorem, it can be shown that there exists $\phi_{R}\in \Sob^{2,p}(\sB_R)$ satisfying the Dirichlet problem 
\begin{align*}
\min_{\zeta \in\Act}\left[\sL_{\zeta}\phi_{R}(x) + c(x, \zeta)\right] = \alpha \phi_{R}(x) \,,\quad \text{for all\ }\,\, x\in\sB_R\,,\quad\text{with}\quad
 \phi_{R} = 0\,\,\, \text{on}\,\,\, \partial{\sB_{R}}\,.
\end{align*} Now letting $R\to\infty$ and following \cite[Theorem~3.5.6]{ABG-book} we arrive at the solution. 

Hence, one can replace our assumption (A4) by the following (relatively weaker) assumption
\begin{itemize}
\item[\hypertarget{A4\'}{{(A4\')}}]
The \emph{running cost} $c$ is bounded (i.e., $\|c\|_{\infty} \leq M$ for some positive constant $M$) and jointly continuous in both variables $(x, \zeta)$\,.
\end{itemize}
All the results of this paper will also hold if we replace (A4) by (A4\')\,. 
\end{remark}
As in Theorem~\ref{TD1.1}, following \cite[Theorem~3.5.6]{ABG-book}, for each approximating model we have the following complete characterization of an optimal policy, which is in the space of stationary Markov policies. 
\begin{theorem}\label{TD1.2}
Suppose (A5)(iii) holds. Then for each $n\in \NN$, there exists a unique solution $V_{\alpha}^n\in\cC^2(\Rd)\cap\cC_b(\Rd)$ of
\begin{equation}\label{APOptDHJB1}
\min_{\zeta \in\Act}\left[\sL_{\zeta}^nV_{\alpha}^n(x) + c_n(x, \zeta)\right] = \alpha V_{\alpha}^n(x) \,,\quad \text{for all\ }\,\, x\in\Rd\,.
\end{equation}
Moreover, we have the following:
\begin{itemize}
\item[(i)] $V_{\alpha}^n$ is the optimal discounted cost, i.e.,
\begin{equation*}
V_{\alpha}^n(x) = \inf_{U\in \Uadm}\Exp_x^{U} \left[\int_0^{\infty} e^{-\alpha t} c_n(X_s^n, U_s) \D s\right]\quad x\in\Rd\,,
\end{equation*}
\item[(ii)]$v_n^*\in \Usm$ is $\alpha$-discounted optimal control if and only if it is a measurable minimizing selector of\cref{APOptDHJB1}, i.e.,
\begin{equation}\label{OtpHJBSelc1}
b_n(x,v_n^*(x))\cdot \grad V_{\alpha}^n(x) + c_n(x, v_n^*(x)) = \min_{\zeta\in \Act}\left[ b_n(x, \zeta)\cdot \grad V_{\alpha}^n(x) + c_n(x, \zeta)\right]\quad \text{a.e.}\,\,\, x\in\Rd\,.
\end{equation}
\end{itemize}
\end{theorem}
In the next theorem, we prove that $V_{\alpha}^n(x)$ converges to $V_{\alpha}(x)$ as $n\to \infty$ for all $x\in\Rd$\,. This result will be useful in establishing the robustness of discounted optimal controls.
\begin{theorem}\label{TC1.3}
Suppose Assumptions (A1)-(A5) hold. Then 
\begin{equation}\label{EC1.1}
\lim_{n\to\infty} V_{\alpha}^n(x) = V_{\alpha}(x) \quad\text{for all}\,\, x\in\Rd\,.
\end{equation}
\end{theorem}

\begin{proof}
From \cref{APOptDHJB1} and \cref{OtpHJBSelc1} for any minimizing selector $v_n^*\in \Usm$, it follows that
\begin{equation*}
\trace\bigl(a_n(x)\grad^2 V_{\alpha}^n(x)\bigr) + b_n(x,v_n^*(x))\cdot \grad V_{\alpha}^n(x) + c_n(x, v_n^*(x)) = \alpha V_{\alpha}^n(x)\,.
\end{equation*}
Then using the standard elliptic PDE estimate as in \cite[Theorem~9.11]{GilTru}, for any $p\geq d+1$ and $R >0$, we deduce that
\begin{equation}\label{ETC1.3A}
\norm{V_{\alpha}^n(x)}_{\Sob^{2,p}(\sB_R)}
\,\le\, \kappa_1\bigl(\norm{V_{\alpha}^n(x)}_{L^p(\sB_{2R})} + \norm{c_n(x, v_n^*(x))}_{L^p(\sB_{2R})}\bigr)\,,
\end{equation}
where $\kappa_1$ is a positive constant which is independent of $n$\,. Since 
\begin{equation*}
\norm{c_n}_{\infty} \,\df\, \sup_{(x,u)\in\Rd\times\Act} c_n(x,u) \leq M, \quad \text{and}\quad V_{\alpha}^n(x) \leq \frac{\norm{c_n}_{\infty}}{\alpha}\,,
\end{equation*} from \cref{ETC1.3A} we get
\begin{equation}\label{ETC1.3B}
\norm{V_{\alpha}^n(x)}_{\Sob^{2,p}(\sB_R)}
\,\le\, \kappa_1 M\bigl(\frac{|\sB_{2R}|^{\frac{1}{p}}}{\alpha} + |\sB_{2R}|^{\frac{1}{p}}\bigr)\,.
\end{equation}We know that for $1< p < \infty$, the space $\Sob^{2,p}(\sB_R)$ is reflexive and separable, hence, as a corollary of the Banach-Alaoglu theorem, we have that every bounded sequence in $\Sob^{2,p}(\sB_R)$ has a weakly convergent subsequence (see, \cite[Theorem~3.18.]{HB-book}). Also, we know that for $p\geq d+1$ the space $\Sob^{2,p}(\sB_R)$ is compactly embedded in $\cC^{1, \beta}(\bar{\sB}_R)$\,, where $\beta < 1 - \frac{d}{p}$ (see \cite[Theorem~A.2.15 (2b)]{ABG-book}), which implies that every weakly convergent sequence in $\Sob^{2,p}(\sB_R)$ will converge strongly in $\cC^{1, \beta}(\bar{\sB}_R)$\,. Thus, in view of estimate \cref{ETC1.3B}, by a standard diagonalization argument and the Banach-Alaoglu theorem, we can extract a subsequence $\{V_{\alpha}^{n_k}\}$ such that for some $V_{\alpha}^*\in \Sobl^{2,p}(\Rd)$
\begin{equation}\label{ETC1.3BC}
\begin{cases}
V_{\alpha}^{n_k}\to & V_{\alpha}^*\quad \text{in}\quad \Sobl^{2,p}(\Rd)\quad\text{(weakly)}\\
V_{\alpha}^{n_k}\to & V_{\alpha}^*\quad \text{in}\quad \cC^{1, \beta}_{loc}(\Rd) \quad\text{(strongly)}\,.
\end{cases}       
\end{equation} 
In the following, we will show that $V^*_{\alpha} = V_{\alpha}$. Now, for any compact set $K\subset \Rd$, it is easy to see that
\begin{align}\label{ETC1.3C1A}
&\max_{x\in K}|\min_{\zeta\in \Act}\{ b_{n_k}(x,\zeta)\cdot \grad V_{\alpha}^{n_k}(x) + c(x, \zeta)\} - \min_{\zeta\in \Act}\{b(x,\zeta)\cdot \grad V_{\alpha}^*(x) + c(x, \zeta)\}|\nonumber\\
&\, \leq \, \max_{x\in K}\max_{\zeta\in \Act}|\{ b_{n_k}(x,\zeta)\cdot \grad V_{\alpha}^{n_k}(x) + c_n(x, \zeta)\} - \{b(x,\zeta)\cdot \grad V_{\alpha}^*(x) + c(x, \zeta)\}| \nonumber\\
&\, \leq \, \max_{x\in K}\max_{\zeta\in \Act}| b_{n_k}(x,\zeta)\cdot \grad V_{\alpha}^{n_k}(x) -b(x,\zeta)\cdot \grad V_{\alpha}^*(x) | + \max_{x\in K}\max_{\zeta\in \Act}| c_{n_k}(x, \zeta) - c(x, \zeta)|
\end{align}
Since $c_{n}(x, \cdot)\to c(x, \cdot)$, $b_{n}(x, \cdot)\to b(x, \cdot)$ continuously on compact set $\Act$ and $V_{\alpha}^{n_k}\to V_{\alpha}^*$ in $\cC^{1, \beta}_{loc}(\Rd)\,,$ for any compact set $K\subset \Rd$, as $k\to \infty$ we deduce that 
\begin{align}\label{ETC1.3C}
\max_{x\in K}|\min_{\zeta\in \Act}\{ b_{n_k}(x,\zeta)\cdot \grad V_{\alpha}^{n_k}(x) + c_{n_k}(x, \zeta)\} - \min_{\zeta\in \Act}\{b(x,\zeta)\cdot \grad V_{\alpha}(x) + c(x, \zeta)\}|\to 0
\end{align}
Thus, multiplying by a test function $\phi\in \cC_{c}^{\infty}(\Rd)$, from \cref{APOptDHJB1}, we obtain
\begin{equation*}
\int_{\Rd}\trace\bigl(a_{n_k}(x)\grad^2 V_{\alpha}^{n_k}(x)\bigr)\phi(x)\D x + \int_{\Rd} \min_{\zeta\in \Act} \{b_{n_k}(x,\zeta)\cdot \grad V_{\alpha}^{n_k}(x) + c_{n_k}(x, \zeta)\}\phi(x)\D x = \alpha\int_{\Rd} V_{\alpha}^{n_k}(x)\phi(x)\D x\,.
\end{equation*}
In view of \cref{ETC1.3BC} and \cref{ETC1.3C}, letting $k\to\infty$ it follows that
\begin{equation}\label{ETC1.3D}
\int_{\Rd}\trace\bigl(a(x)\grad^2 V_{\alpha}^*(x)\bigr)\phi(x)\D x + \int_{\Rd} \min_{\zeta\in \Act} \{b(x,\zeta)\cdot \grad V_{\alpha}^*(x) + c(x, \zeta)\}\phi(x)\D x = \alpha\int_{\Rd} V_{\alpha}^*(x)\phi(x)\D x\,.
\end{equation} Since $\phi\in \cC_{c}^{\infty}(\Rd)$ is arbitrary and $V_{\alpha}^*\in \Sobl^{2,p}(\Rd)$ from \cref{{ETC1.3D}} we deduce that
\begin{equation}\label{ETC1.3E}
\min_{\zeta \in\Act}\left[\sL_{\zeta}V_{\alpha}^*(x) + c(x, \zeta)\right] = \alpha V_{\alpha}^*(x) \,,\quad \text{a.e.\ }\,\, x\in\Rd\,.
\end{equation}
Let $\tilde{v}^*\in\Usm$ be a minimizing selector of \cref{ETC1.3E} and $\tilde{X}$ be the solution of the SDE \cref{E1.1} corresponding to $\tilde{v}^*$. Then applying It$\hat{\rm o}$-Krylov formula, we obtain the following
\begin{align*}
& \Exp_x^{\tilde{v}^*}\left[ e^{-\alpha T} V_{\alpha}^{*}(\tilde{X}_{T})\right] - V_{\alpha}^{*}(x)\nonumber\\ 
& \,=\,\Exp_x^{\tilde{v}^*}\left[\int_0^{T} e^{-\alpha s}\{\trace\bigl(a(\tilde{X}_s)\grad^2 V_{\alpha}^{*}(\tilde{X}_s)\bigr) + b(\tilde{X}_s, \tilde{v}^*(\tilde{X}_s))\cdot \grad V_{\alpha}^{*}(\tilde{X}_s) - \alpha V_{\alpha}^{*}(\tilde{X}_s))\} \D{s}\right] \,.
\end{align*}
Hence, using \cref{ETC1.3E}, we deduce that
\begin{align}\label{ETC1.3FC}
\Exp_x^{\tilde{v}^*}\left[ e^{-\alpha T} V_{\alpha}^{*}(\tilde{X}_{T})\right] - V_{\alpha}^{*}(x) \,=\,- \Exp_x^{\tilde{v}^*}\left[\int_0^{T} e^{-\alpha s}c(\tilde{X}_s, \tilde{v}^*(\tilde{X}_s))\D{s}\right] \,.
\end{align}
Since $V_{\alpha}^{*}$ is bounded and
$$\Exp_x^{\tilde{v}^*}\left[ e^{-\alpha T} V_{\alpha}^{*}(\tilde{X}_{T})\right] = e^{-\alpha T}\Exp_x^{\tilde{v}^*}\left[ V_{\alpha}^{*}(\tilde{X}_{T})\right],$$
letting $T\to\infty$, it is easy to see that
\begin{equation*}
\lim_{T\to\infty}\Exp_x^{\tilde{v}^*}\left[ e^{-\alpha T} V_{\alpha}^{*}(\tilde{X}_{T})\right] = 0\,.
\end{equation*}
Now, letting $T \to \infty$ by monotone convergence theorem, from \cref{ETC1.3FC} we obtain 
\begin{align}\label{ETC1.3FD}
 V_{\alpha}^{*}(x) \,=\, \Exp_x^{\tilde{v}^*}\left[\int_0^{\infty} e^{-\alpha s}c(\tilde{X}_s, \tilde{v}^*(\tilde{X}_s)) \D{s}\right] \,.
\end{align}
Again by similar argument, applying It$\hat{\rm o}$-Krylov formula and using \cref{ETC1.3E}, for any $U\in \Uadm$\,, we have
\begin{align*}
 V_{\alpha}^{*}(x) \,
 \leq\, \Exp_x^{U}\left[\int_0^{\infty} e^{-\alpha s}c(\tilde{X}_s, U_s) \D{s}\right] \,.
\end{align*}This implies
\begin{align}\label{ETC1.3FE}
 V_{\alpha}^{*}(x) \,
 \leq\,\inf_{U\in\Uadm} \Exp_x^{U}\left[\int_0^{\infty} e^{-\alpha s}c(\tilde{X}_s, U_s) \D{s}\right] \,.
\end{align}
 Thus, from \cref{ETC1.3FD} and \cref{ETC1.3FE}, we deduce that
\begin{align}\label{ETC1.3FF}
 V_{\alpha}^{*}(x) \,=
 \, \inf_{U\in\Uadm}\Exp_x^{U}\left[\int_0^{\infty} e^{-\alpha s}c(\tilde{X}_s, U_s) \D{s}\right] \,.
\end{align}
Since both $V_{\alpha}, V_{\alpha}^*$ are continuous functions on $\Rd$, from \cref{OPDcost} and \cref{ETC1.3FF}, it follows that $V_{\alpha}(x) = V_{\alpha}^*(x)$ for all $x\in\Rd$. This completes the proof.
\end{proof}
Let $\hat{X}^n$ be the solution of the SDE \cref{E1.1} corresponding to $v_n^*$. Then we have
\begin{equation}\label{EDiscostRobust}
\cJ_{\alpha}^{v_n^*}(x, c) \,=\, \Exp_x^{v_n^*} \left[\int_0^{\infty} e^{-\alpha t} c(\hat{X}^n_s, v_n^*(\hat{X}^n_s)) \D s\right],\quad x\in\Rd\,.
\end{equation}
Next we prove the robustness result, i.e., we prove that $\cJ_{\alpha}^{v_n^*}(x, c) \to \cJ_{\alpha}^{v^*}(x, c)$ as $n\to \infty$\,, where $v_n^*$ is an optimal control of the approximated model and $v^*$ is an optimal control of the true model. As in \cite{KY-20} we will use the continuity result above as an intermediate step.
\begin{theorem}\label{TC1.4}
Suppose Assumptions (A1)-(A5) hold. Then 
\begin{equation}\label{ETC1.4A}
\lim_{n\to\infty} \cJ_{\alpha}^{v_n^*}(x, c) = \cJ_{\alpha}^{v^*}(x, c) \quad\text{for all}\,\, x\in\Rd\,.
\end{equation}
\end{theorem}
\begin{proof}
Following the argument as in \cite[Theorem~3.5.6]{ABG-book}, one can show that for each $v_n^*\in \Usm$, there exists $V_{\alpha}^{n,*}\in \Sobl^{2,p}(\Rd)\cap \cC_{b}(\Rd)$ satisfying 
\begin{equation}\label{ETC1.4B}
\trace\bigl(a(x)\grad^2 V_{\alpha}^{n,*}(x)\bigr) + b(x,v_n^*(x))\cdot \grad V_{\alpha}^{n,*}(x) + c(x, v_n^*(x)) = \alpha V_{\alpha}^{n,*}(x)\,.
\end{equation}
Applying It$\hat{\rm o}$-Krylov formula, we deduce that
\begin{align*}
& \Exp_x^{v_n^*}\left[ e^{-\alpha T} V_{\alpha}^{n,*}(\hat{X}^n_{T})\right] - V_{\alpha}^{n,*}(x)\nonumber\\ 
& \,=\,\Exp_x^{v_n^*}\left[\int_0^{T} e^{-\alpha s}\{\trace\bigl(a(\hat{X}^n_s)\grad^2 V_{\alpha}^{n,*}(\hat{X}^n_s)\bigr) + b(\hat{X}^n_s, v_n^*(\hat{X}^n_s))\cdot \grad V_{\alpha}^{n,*}(\hat{X}^n_s) - \alpha V_{\alpha}^{n,*}(\hat{X}^n_s))\} \D{s}\right] \,.
\end{align*}
Now using \cref{ETC1.4B}, it follows that
\begin{align}\label{ETC1.4C}
\Exp_x^{v_n^*}\left[ e^{-\alpha T} V_{\alpha}^{n,*}(\hat{X}^n_{T})\right] - V_{\alpha}^{n,*}(x) \,=\,- \Exp_x^{v_n^*}\left[\int_0^{T} e^{-\alpha s}c(\hat{X}^n_s, v_n^*(\hat{X}^n_s)) \D{s}\right] \,.
\end{align}
Since $V_{\alpha}^{n,*}$ is bounded and
$$\Exp_x^{v_n^*}\left[ e^{-\alpha T} V_{\alpha}^{n,*}(\hat{X}^n_{T})\right] = e^{-\alpha T}\Exp_x^{v_n^*}\left[ V_{\alpha}^{n,*}(\hat{X}^n_{T})\right],$$
letting $T\to\infty$ we deduce that
\begin{equation*}
\lim_{T\to\infty}\Exp_x^{v_n^*}\left[ e^{-\alpha T} V_{\alpha}^{n,*}(\hat{X}^n_{T})\right] = 0\,.
\end{equation*}
Thus, from \cref{ETC1.4C}, letting $T \to \infty$ by monotone convergence theorem we obtain
\begin{align}\label{ETC1.4D}
 V_{\alpha}^{n,*}(x) \,=\, \Exp_x^{v_n^*}\left[\int_0^{\infty} e^{-\alpha s}c(\hat{X}^n_s, v_n^*(\hat{X}^n_s)) \D{s}\right] \,=\, \cJ_{\alpha}^{v_n^*}(x, c)\,.
\end{align}
This implies that $V_{\alpha}^{n,*} \leq \frac{\norm{c}_{\infty}}{\alpha}$. Thus, as in Theorem \ref{TC1.3} (see, \cref{ETC1.3A}, \cref{ETC1.3B}), by standard Sobolev estimate, for any $R>0$ we get
$\norm{V_{\alpha}^{n,*}}_{\Sob^{2,p}(\sB_R)} \leq \kappa_2$\,, for some positive constant $\kappa_2$ independent of $n$. Hence, by the Banach-Alaoglu theorem and standard diagonalization argument (as in \cref{ETC1.3BC}), we have there exists $\hat{V}_{\alpha}^*\in \Sobl^{2,p}(\Rd)\cap \cC_{b}(\Rd)$ such that along some sub-sequence $\{V_{\alpha}^{n_{k},*}\}$ 
\begin{equation}\label{ETC1.4E}
\begin{cases}
V_{\alpha}^{n_{k},*}\to & \hat{V}_{\alpha}^*\quad \text{in}\quad \Sobl^{2,p}(\Rd)\quad\text{(weakly)}\\
V_{\alpha}^{n_{k},*}\to & \hat{V}_{\alpha}^*\quad \text{in}\quad \cC^{1, \beta}_{loc}(\Rd)\quad\text{(strongly)}\,.
\end{cases}       
\end{equation} Since space of stationary Markov strategies $\Usm$ is compact,  along some further sub-sequence (without loss of generality denoting by same sequence) we have $v_{n_k}^*\to \hat{v}^*$ in $\Usm$\,. It is easy to see that
\begin{align*}
b(x,v_{n_k}^*(x))\cdot \grad V_{\alpha}^{n_{k}, *}(x) - b(x,\hat{v}^*(x))\cdot \grad \hat{V}_{\alpha}^*(x) = & b(x,v_{n_k}^*(x))\cdot \grad \left(V_{\alpha}^{n_{k}, *} - \hat{V}_{\alpha}^*\right)(x) \\
& + \left(b(x,v_{n_k}^*(x)) - b(x,\hat{v}^*(x))\right)\cdot \grad \hat{V}_{\alpha}^*(x)\,.
\end{align*}
Since $V_{\alpha}^{n_k, *}\to \hat{V}_{\alpha}^*$ in $\cC^{1, \beta}_{loc}(\Rd)\,,$ on any compact set $b(x,v_{n_k}^*(x))\cdot \grad \left(V_{\alpha}^{n_{k}, *} - \hat{V}_{\alpha}^*\right)(x)\to 0$ strongly and by the topology of $\Usm$, we have $\left(b(x,v_{n_k}^*(x)) - b(x,\hat{v}^*(x))\right)\cdot \grad \hat{V}_{\alpha}^*(x)\to 0$ weakly. Thus, in view of the topology of $\Usm$, and since $V_{\alpha}^{n_k, *}\to \hat{V}_{\alpha}^*$ in $\cC^{1, \beta}_{loc}(\Rd)\,,$ as $k\to \infty$ we obtain 
\begin{equation}\label{ETC1.4EA}
b(x,v_{n_k}^*(x))\cdot \grad V_{\alpha}^{n_{k}, *}(x) + c(x, v_{n_k}^*(x)) \to b(x,\hat{v}^*(x))\cdot \grad \hat{V}_{\alpha}^*(x) + c(x, \hat{v}^*(x))\quad\text{weakly}\,.
\end{equation}
Now, multiplying by a test function $\phi\in \cC_{c}^{\infty}(\Rd)$, from \cref{ETC1.4B}, it follows that
\begin{align*}
\int_{\Rd}\trace\bigl(a(x)\grad^2 V_{\alpha}^{n_{k}, *}(x)\bigr)\phi(x)\D x + \int_{\Rd}\{b(x,v_{n_k}^*(x))\cdot \grad V_{\alpha}^{n_{k}, *}(x) + & c(x, v_{n_k}^*(x))\}\phi(x)\D x \\
&= \alpha\int_{\Rd} V_{\alpha}^{n_{k}, *}(x)\phi(x)\D x\,.
\end{align*}
Hence, using \cref{ETC1.4E}, \cref{ETC1.4EA}, and letting $k\to\infty$ we obtain
\begin{equation}\label{ETC1.4EB}
\int_{\Rd}\trace\bigl(a(x)\grad^2 \hat{V}_{\alpha}^*(x)\bigr)\phi(x)\D x + \int_{\Rd} \{b(x,\hat{v}^*(x))\cdot \grad \hat{V}_{\alpha}^*(x) + c(x, \hat{v}^*(x))\}\phi(x)\D x = \alpha\int_{\Rd} \hat{V}_{\alpha}^*(x)\phi(x)\D x\,.
\end{equation} Since $\phi\in \cC_{c}^{\infty}(\Rd)$ is arbitrary and $\hat{V}_{\alpha}^*\in \Sobl^{2,p}(\Rd)$ from \cref{ETC1.4EB}, we deduce that
the function $\hat{V}_{\alpha}^*\in \Sobl^{2,p}(\Rd)\cap \cC_{b}(\Rd)$ satisfies
\begin{equation}\label{ETC1.4F}
\trace\bigl(a(x)\grad^2 \hat{V}_{\alpha}^{*}(x)\bigr) + b(x,\hat{v}^*(x))\cdot \grad \hat{V}_{\alpha}^{*}(x) + c(x, \hat{v}^*(x)) = \alpha \hat{V}_{\alpha}^{*}(x)\,.
\end{equation}
As earlier, applying It$\hat{\rm o}$-Krylov formula and using \cref{{ETC1.4F}}, it follows that
\begin{align}\label{ETC1.4G}
 \hat{V}_{\alpha}^{*}(x) \,=\, \Exp_x^{\hat{v}^*}\left[\int_0^{\infty} e^{-\alpha s}c(\hat{X}_s, \hat{v}^*(\hat{X}_s)) \D{s}\right] \,,
\end{align} where $\hat{X}$ is the solution of SDE \cref{E1.1} corresponding to $\hat{v}^*$\,.
 
Now, we have
\begin{equation}
|\cJ_{\alpha}^{v_{n_k}^*}(x, c) - \cJ_{\alpha}^{v^*}(x, c)| \leq |\cJ_{\alpha}^{v_{n_k}^*}(x, c) - V_{\alpha}^{n_k}(x)| + |V_{\alpha}^{n_k}(x) - \cJ_{\alpha}^{v^*}(x, c)|\,.
\end{equation} From Theorem \ref{TD1.1}, we know that $\cJ_{\alpha}^{v^*}(x, c) = V_{\alpha}(x)$. Thus from Theorem~\ref{TC1.3}, we deduce that $|V_{\alpha}^{n_k}(x) - \cJ_{\alpha}^{v^*}(x, c)|\to 0$ as $k\to\infty$\,. To complete the proof we have to show that $|\cJ_{\alpha}^{v_{n_k}^*}(x, c) - V_{\alpha}^{n_k}(x)|\to 0$ as $k\to \infty$\,. 
Also, from Theorem~\ref{TD1.2} we know that $v_n^*\in \Usm$ is a minimizing selector of the HJB equation \cref{APOptDHJB1} of the approximated model, thus it follows that
\begin{align}\label{ETC1.4GA1}
\alpha V_{\alpha}^{n_k}(x) &\,=\, \min_{\zeta \in\Act}\left[\sL_{\zeta}^{n_k}V_{\alpha}^{n_k}(x) + c_{n_k}(x, \zeta)\right]\nonumber\\
& \,=\, \trace\bigl(a_{n_k}(x)\grad^2 V_{\alpha}^{n_k}(x)\bigr) + b_{n_k}(x,v_{n_k}^*(x))\cdot \grad V_{\alpha}^{n_k}(x) + c(x, v_{n_k}^*(x))\,,\quad \text{a.e.\ }\,\, x\in\Rd\,.
\end{align}
Hence, by standard Sobolev estimate (as in Theorem~\ref{TC1.3}), for each $R>0$ we have $\norm{V_{\alpha}^{n_k}}_{\Sob^{2,p}(\sB_R)} \leq \kappa_3$\,, for some positive constant $\kappa_3$ independent of $k$. Thus, we can extract a further sub-sequence (without loss of generality denoting by same sequence) such that for some $\Tilde{V}_{\alpha}^*\in \Sobl^{2,p}(\Rd)\cap \cC_{b}(\Rd)$ (as in \cref{ETC1.3BC}) we get
\begin{equation}\label{ETC1.4H}
\begin{cases}
V_{\alpha}^{n_k}\to & \Tilde{V}_{\alpha}^*\quad \text{in}\quad \Sobl^{2,p}(\Rd)\quad\text{(weakly)}\\
V_{\alpha}^{n_k}\to & \Tilde{V}_{\alpha}^*\quad \text{in}\quad \cC^{1, \beta}_{loc}(\Rd)\quad\text{(strongly)}\,.
\end{cases}       
\end{equation} Following the similar steps as in Theorem \ref{TC1.3}, multiplying by test function and letting $k\to \infty$, from \cref{ETC1.4GA1} we deduce that $\Tilde{V}_{\alpha}^*\in \Sobl^{2,p}(\Rd)\cap \cC_{b}(\Rd)$ satisfies
\begin{align}\label{ETC1.4I}
\alpha \Tilde{V}_{\alpha}^{*}(x) &\,=\, \min_{\zeta \in\Act}\left[\sL_{\zeta}\Tilde{V}_{\alpha}^{*}(x) + c(x, \zeta)\right]\nonumber\\
& \,=\, \trace\bigl(a(x)\grad^2 \Tilde{V}_{\alpha}^{*}(x)\bigr) + b(x,\hat{v}^*(x))\cdot \grad \Tilde{V}_{\alpha}^{*}(x) + c(x, \hat{v}^*(x))(x)\,.
\end{align}
From the continuity results (Theorem~\ref{TC1.3}), it is easy to see that $\Tilde{V}_{\alpha}^{*}(x) = \cJ_{\alpha}^{v^*}(x, c)$ for all $x\in\Rd$. Moreover, applying It$\hat{\rm o}$-Krylov formula and using \cref{{ETC1.4I}} we obtain
\begin{align}\label{ETC1.4J}
 \Tilde{V}_{\alpha}^{*}(x) \,=\, \Exp_x^{\hat{v}^*}\left[\int_0^{\infty} e^{-\alpha s}c(\hat{X}_s, \hat{v}^*(\hat{X}_s)) \D{s}\right] \,.
\end{align}
Since both $\hat{V}_{\alpha}^{*}$, $\Tilde{V}_{\alpha}^{*}$ are continuous, from \cref{ETC1.4G} and \cref{ETC1.4J}, it follows that both $\cJ_{\alpha}^{v_{n_k}^*}(x, c)$ (which is equals to $V_{\alpha}^{n_{k},*}(x)$) and $V_{\alpha}^{n_{k}}(x)$ converge to the same limit. This completes the proof.
\end{proof}
\begin{remark}
Note that in the above, we indirectly also showed the continuity of the value function in the control policy (under the topology defined); uniqueness of the solution to the PDE in the above implies continuity. This result, while can be obtained from the analysis of Borkar \cite{Bor89} (in a slightly more restrictive setup), is obtained directly via a careful optimality analysis and will have important consequences in numerical solutions and approximation results for both discounted and average cost optimality. This is studied in details, with implications in \cite{YukselPradhan}\,.
\end{remark}
\section{Analysis of Ergodic Cost}\label{secErgodicCost}
In this section we study the robustness problem for the ergodic cost criterion. The associated optimal control problem for this cost criterion has been studied extensively in the literature, see e.g., \cite{ABG-book}. 

For this cost evolution criterion we will study the robustness problem under two sets of assumptions: the first is so called near-monotonicity condition on the running cost which discourage instability and second is Lyapunov stability.
\subsection{Analysis under a near-monotonicity assumption}\label{NearMonotone}    
Here we assume that the cost function $c$ satisfies the following near-monotonicity condition:
\begin{itemize}
\item[\hypertarget{A6}{{(A6)}}] It holds that
\begin{equation}\label{ENearmonot}
\liminf_{\norm{x}\to\infty}\inf_{\zeta\in \Act} c(x,\zeta) > \sE^*(c)\,. 
\end{equation}
\end{itemize}
This condition penalizes the escape of probability mass to infinity. Since our running cost $c$ is bounded it is easy to see that $\sE^*(c) \leq \norm{c}_{\infty}$\,. Recall that a stationary policy $v\in \Usm$ is said to be stable if the associated diffusion process is positive recurrent. It is known that under \cref{ENearmonot}, optimal control exists in the space of stable stationary Markov controls (see, \cite[Theorem~3.4.5]{ABG-book}).  

Now from \cite[Theorem~3.6.10]{ABG-book}, we have the following complete characterization of ergodic optimal control.
\begin{theorem}\label{ergodicnearmono1}
Suppose that Assumptions (A1)-(A4) and (A6) hold. Then there exists a unique solution pair $(V, \rho)\in \Sobl^{2,p}(\Rd)\times \RR$, \, $1< p < \infty$, with $V(0) = 0$ and $\inf_{\Rd} V > -\infty$ and $\rho \leq \sE^*(c)$, satisfying
\begin{equation}\label{EErgonearOpt1A}
\rho = \min_{\zeta \in\Act}\left[\sL_{\zeta}V(x) + c(x, \zeta)\right]\,.
\end{equation}
Moreover, we have
\begin{itemize}
\item[(i)]$\rho = \sE^*(c)$
\item[(ii)] a stationary Markov control $v^*\in \Usm$ is an optimal control if and only if it is a minimizing selector of \cref{EErgonearOpt1A}, i.e., if and only if it satisfies
\begin{equation}\label{EErgonearOpt1B}
\min_{\zeta \in\Act}\left[\sL_{\zeta}V(x) + c(x, \zeta)\right] \,=\, \trace\bigl(a(x)\grad^2 V(x)\bigr) + b(x,v^*(x))\cdot \grad V(x) + c(x, v^*(x))\,,\quad \text{a.e.}\,\, x\in\Rd\,.
\end{equation}
\end{itemize}
\end{theorem} 
We assume that for the approximated model, for each $n\in \NN$ the running cost function $c_n$ satisfies the near-monotonicity condition \cref{ENearmonot} relative to $\max_{n\in\NN}\sE^{n*}(c_n)$\, i.e., 
\begin{equation}\label{AssumNearApprox1}
\max_{n\in\NN}\sE^{n*}(c_n) < \liminf_{\norm{x}\to\infty}\inf_{\zeta\in\Act} c_n(x,\zeta)\,.
\end{equation}
Thus, in view of \cite[Theorem~3.6.10]{ABG-book}, for the approximating model, for each $n\in \NN$ we have the following theorem.
\begin{theorem}\label{ergodicnearmono2}
Suppose that Assumption (A5)(iii) holds. Then there exists a unique solution pair $(V_n, \rho_n)\in \Sobl^{2,p}(\Rd)\times \RR$, \, $1< p < \infty$, with $V_n(0) = 0$ and $\inf_{\Rd} V_n > -\infty$ and $\rho_n \leq \sE^{n*}(c_n)$,  satisfying
\begin{equation}\label{EErgonearOpt2A}
\rho_n = \min_{\zeta \in\Act}\left[\sL_{\zeta}^{n}V_n(x) + c_n(x, \zeta)\right]
\end{equation}
Moreover, we have
\begin{itemize}
\item[(i)]$\rho_n = \sE^{n*}(c_n)$
\item[(ii)] a stationary Markov control $v_n^*\in \Usm$ is an optimal control if and only if it is a minimizing selector of \cref{EErgonearOpt2A}, i.e., if and only if it satisfies
\begin{equation}\label{EErgonearOpt2B}
\min_{\zeta \in\Act}\left[\sL_{\zeta}^n V_n(x) + c_n(x, \zeta)\right] \,=\, \trace\bigl(a_n(x)\grad^2 V_n(x)\bigr) + b_n(x,v_n^*(x))\cdot \grad V_n(x) + c_n(x, v_n^*(x))\,,\quad \text{a.e.}\,\, x\in\Rd\,.
\end{equation}
\end{itemize}
\end{theorem}  
In view of the near-monotonicity assumption \cref{AssumNearApprox1}, for any minimizing selector $v_n^*\in\Usm$ of \cref{EErgonearOpt2A}, it is easy to see that outside a compact set $\sL_{v_n^*}^{n}V_n(x) \leq -\epsilon$ for some $\epsilon>0$\,. Since $V_n$ is bounded from below, \cite[Theorem~2.6.10(f)]{ABG-book} asserts that $v_n^*$ is stable. Hence, we deduce that the optimal policies of the approximating models are stable. However, note that the compact set mentioned above may not be applicable uniformly for all $n$, which turns out to be a consequential issue.

Now we want to show that as $n\to\infty$ the optimal value of the approximated model $\sE^{n*}(c_n)$ converges to the optimal value of the true model $\sE^{*}(c)$\,. Under near-monotonicity assumption this result may not be true in general due to the restricted uniqueness/non-uniqueness of the solution of the associated HJB equation (see e.g., \cite{AA12}, \cite{AA13}). As a result of this, in \cite{AA12}, \cite{M97} the authors have shown that for the optimal control problem the policy iteration algorithm (PIA) may fail to converge to the optimal value. In order to to ensure convergence of the PIA, in addition to the near-monotonicity assumption a blanket Lyapunov condition is assumed in \cite{M97}\,. 

Accordingly, in this article, to guarantee the convergence $\sE^{n*}(c_n)\to \sE^{*}(c)$, we will assume that 
\[\Theta \,\df\, \{\eta_{v_n^*}^n : n\in\NN\},\] is tight, where $\eta_{v_n^*}^n$ is the unique invariant measure of the solution $X^n$ of \cref{ASE1.1} corresponding to $v_n^*\in \Usm$ (the optimal policies of the approximated models)\,. One sufficient condition which ensures the required tightness is the following: there exists a pair of nonnegative inf-compact functions $(\Lyap, h)\in \cC^{2}(\Rd)\times\cC(\Rd)$ such that $\sL_{v_n^*}^{n} \Lyap(x) \leq \hat{\kappa}_{0} - h(x)$ for some positive constant $\hat{\kappa}_{0}$ and for all $n\in\NN$ and $x\in \Rd$\,.
\begin{theorem}\label{ErgoContnuity}
Suppose that Assumptions (A1) - (A6) hold. Also, assume that the set $\Theta$ is  tight. Then, we have
\begin{equation}\label{EErgoContnuity1}
\lim_{n\to\infty} \sE^{n*}(c_n) = \sE^{*}(c)\,.
\end{equation}
\end{theorem}
\begin{proof}
From Theorem~\ref{ergodicnearmono2}, we know that for each $n\in\NN$ there exists $(V_n, \rho_n)\in \Sobl^{2,p}(\Rd)\times \RR$, \, $1< p < \infty$, with $V_n(0) = 0$ and $\inf_{\Rd} V_n > -\infty$, satisfying
\begin{equation}\label{EErgoContnuity1A}
\rho_n = \min_{\zeta \in\Act}\left[\sL_{\zeta}^{n}V_n(x) + c_n(x, \zeta)\right]\,,
\end{equation} where $\rho_n = \sE^{n*}(c_n)$\,. Since $\norm{c_n}_{\infty} \leq M$, it follows that $\rho_n = \sE^{n*}(c_n) \leq M$\,.  

From \cite[Theorem~3.6.6]{ABG-book} (the standard vanishing discount asymptotics), we know that as $\alpha \to 0$ the difference $V_{\alpha}^n(\cdot) - V_{\alpha}^n(0) \to V_{n}(\cdot)$ and $\alpha V_{\alpha}^n(0)\to \rho_n$, where $V_{\alpha}^n$ is the solution of the $\alpha$-discounted HJB equation 
\cref{APOptDHJB1}\,. Let $$\kappa(\rho_n) \,\df\, \{x\in\Rd\mid \min_{\zeta\in\Act} c_n(x,\zeta) \leq \rho_n\}\,.$$ Since the map $x\to \min_{\zeta\in\Act} c_n(x,\zeta)$ is continuous, it is easy to see that $\kappa(\rho_n)$ is closed and due the near-monotonicity assumption (see, \cref{AssumNearApprox1}), it follows that $\kappa(\rho_n)$ is bounded. Therefore $\kappa(\rho_n)$ is a compact subset of $\Rd$. Since $V_{\alpha}^{n} \leq \cJ_{\alpha,n}^{v_n^*}(x, c_n)$ and $v_n^*$ is stable, from \cite[Lemma~3.6.1]{ABG-book}, we have 
\begin{equation}\label{EErgoContnuity1C}
\inf_{\kappa(\rho_n)} V_{\alpha}^{n} = \inf_{\Rd} V_{\alpha}^{n} \leq \frac{\rho_n}{\alpha}\,.
\end{equation} Now for any minimizing selector $\hat{v}_n^*\in \Usm$ of \cref{APOptDHJB1}, we get
\begin{equation*}
\trace\bigl(a_n(x)\grad^2 V_{\alpha}^n(x)\bigr) + b_n(x,\hat{v}_n^*(x))\cdot \grad V_{\alpha}^n(x) - \alpha V_{\alpha}^n(x) = - c_n(x, \hat{v}_n^*(x))\,.  
\end{equation*} Since $\norm{c_n}_{\infty} \leq M$ for all $n\in\NN$, from estimate (3.6.9b) of \cite[Lemma~3.6.3]{ABG-book}, it follows that
\begin{equation}\label{EErgoContnuity1D}
\norm{V_{\alpha}^n - V_{\alpha}^n(0)}_{\Sob^{2,p}(\sB_R)} \leq \Tilde{C}(R,p)\left(1 + \alpha \inf_{\sB_{R_0}}V_{\alpha}^n \right)\,,
\end{equation} for all $R> R_0$, where $R_0\in\RR$ is positive number such that $\kappa(\rho_n)\subset\sB_{R_0}$ and $\Tilde{C}(R,p)$ is a positive constant which depends only on $d$ and $R_0$\,. Now combining \cref{EErgoContnuity1C} and \cref{EErgoContnuity1D}, we obtain 
\begin{equation}\label{EErgoContnuity1E}
\norm{V_{\alpha}^n - V_{\alpha}^n(0)}_{\Sob^{2,p}(\sB_R)} \leq \Tilde{C}(R,p)\left(1 + M \right)\,.
\end{equation} In view of assumption \cref{AssumNearApprox1}, one can choose $R_0$ independent of $n$. Thus \cref{EErgoContnuity1E} implies that
\begin{equation}\label{EErgoContnuity1F}
\norm{V_n}_{\Sob^{2,p}(\sB_R)} \leq \Tilde{C}(R,p)\left(1 + M \right)\,.
\end{equation}
Hence, by the Banach-Alaoglu theorem and standard diagonalization argument (as in \cref{ETC1.3BC}), we have there exists $\Tilde{V}\in \Sobl^{2,p}(\Rd)$ such that along a sub-sequence
\begin{equation}\label{EErgoContnuity1G}
\begin{cases}
V_{n_k}\to & \Tilde{V}\quad \text{in}\quad \Sobl^{2,p}(\Rd)\quad\text{(weakly)}\\
V_{n_k}\to & \Tilde{V}\quad \text{in}\quad \cC^{1, \beta}_{loc}(\Rd)\quad\text{(strongly)}\,.
\end{cases}       
\end{equation}
Again, since $\rho_n \leq M$, along a further sub-sequence (without loss of generality denoting by same sequence), we have $\rho_{n_k}\to \Tilde{\rho}$ as $k\to \infty$\,. Now, as before, multiplying by test function $\phi\in \cC_{c}^{\infty}(\Rd)$, from \cref{EErgoContnuity1A}, we obtain
\begin{equation*}
\int_{\Rd}\trace\bigl(a_{n_k}(x)\grad^2 V_{n_k}(x)\bigr)\phi(x)\D x + \int_{\Rd} \min_{\zeta\in \Act} \{b_{n_k}(x,\zeta)\cdot \grad V_{n_k}(x) + c_{n_k}(x, \zeta)\}\phi(x)\D x = \int_{\Rd} \rho_{n_k}\phi(x)\D x\,.
\end{equation*}
By similar argument as in Theorem~\ref{TC1.3}, in view of \cref{EErgoContnuity1G}, letting $k\to\infty$ it follows that
\begin{equation}\label{EErgoContnuity1H}
\int_{\Rd}\trace\bigl(a(x)\grad^2 \Tilde{V}(x)\bigr)\phi(x)\D x + \int_{\Rd} \min_{\zeta\in \Act} \{b(x,\zeta)\cdot \grad \Tilde{V}(x) + c(x, \zeta)\}\phi(x)\D x = \int_{\Rd} \Tilde{\rho}\phi(x)\D x\,.
\end{equation} Since $\phi\in \cC_{c}^{\infty}(\Rd)$ is arbitrary and $\Tilde{V}\in \Sobl^{2,p}(\Rd)$, we deduce that $\Tilde{V}\in \Sobl^{2,p}(\Rd)$ satisfies
\begin{equation*}
\Tilde{\rho} = \min_{\zeta \in\Act}\left[\sL_{\zeta}\Tilde{V}(x) + c(x, \zeta)\right]\,.
\end{equation*} Since $\Usm$ is compact along a further subsequence $v_{n_k}^*\to \tilde{v}^*$ (denoting by the same sequence) in $\Usm$. Repeating the above argument, one can show that the pair $(\Tilde{V}, \Tilde{\rho})$ satisfies 
\begin{equation*}
\Tilde{\rho} = \sL_{\tilde{v}^*}\Tilde{V}(x) + c(x, \tilde{v}^*(x))\,.
\end{equation*}
 As we know $V_n(0) = 0$ for all $n\in\NN$ (see, \cref{EErgoContnuity1A}), it is easy to see that $\Tilde{V}(0) = 0$. Next we show that $\Tilde{V}$ is bounded from below. From estimate (3.6.9a) of \cite[Lemma~3.6.3]{ABG-book}, for each $R> R_0$ we have
\begin{equation}\label{EErgoContnuity1I}
\left(\osc_{\sB_{2R}} V_{\alpha}^{n}\,\df\, \right) \sup_{x\in \sB_{2R}} V_{\alpha}^{n}(x) - \inf_{x\in \sB_{2R}}V_{\alpha}^{n}(x) \leq \Tilde{C}_1(R)(1 + \alpha\inf_{\sB_{R_0}} V_{\alpha}^{n})\leq \Tilde{C}_1(R)(1 + M)\,,
\end{equation} for some constant $\Tilde{C}_1(R) >0$ which depends only on $d$ and $R_0$\,. Also, let $\alpha_k$ be a sequence such that $\alpha_k\to 0$ as $k\to \infty$, thus for each $x\in \Rd$ we have
\begin{align}\label{EErgoContnuity1IA}
V_n(x) &= \lim_{k\to \infty}\left(V_{\alpha_k}^n(x) - V_{\alpha_k}^n(0)\right) \geq \liminf_{k\to \infty} \left(V_{\alpha_k}^n(x) - \inf_{\Rd}V_{\alpha_k}^n(x) + \inf_{\Rd}V_{\alpha_k}^n(x) - V_{\alpha_k}^n(0)\right) \nonumber\\
&\geq - \limsup_{k\to \infty} \left(V_{\alpha_k}^n(0) - \inf_{\Rd}V_{\alpha_k}^n(x)\right) + \liminf _{k\to \infty} \left(V_{\alpha_k}^n(x) - \inf_{\Rd}V_{\alpha_k}^n(x)\right)\nonumber\\
&\geq - \limsup_{k\to \infty} \left(\osc_{\sB_{R_0}} V_{\alpha_k}^n \right);\quad \left(\text{since}\,\,\, \inf_{\sB_{R_0}} V_{\alpha_k}^n = \inf_{\Rd} V_{\alpha_k}^n \right)\,,
\end{align} where the last inequality follows form the fact that $V_{\alpha_k}^n(x) - \inf_{\Rd}V_{\alpha_k}^n(x) \geq 0$\,. Hence, in view of estimate \cref{EErgoContnuity1I}, we deduce that
\begin{equation}\label{EErgoContnuity1J}
V_n(x) \geq - \Tilde{C}_1(R_0)(1 + M)\,.
\end{equation} This implies that the limit $\Tilde{V}\geq - \Tilde{C}_1(R_0)(1 + M)$\,. Note that $$\rho_{n_k} = \sE^{{n_k}*}(c_{n_k}) = \int_{\Rd}\int_{\Act} c_{n_k}(x,\zeta)v_{n_k}^*(x)(\D \zeta)\eta_{v_{n_k}^*}^{n_k}(\D x)\,.$$ Since $\Theta$ is tight, from \cite[Lemma~3.2.6]{ABG-book}, we deduce that $\eta_{v_{n_k}^*}^{n_k} \to \eta_{\tilde{v}^*}$ in total variation norm as $k\to\infty$, where $\eta_{\tilde{v}^*}$ is the unique invariant measure of \cref{E1.1} corresponding to $\tilde{v}^*$\,. Thus, by writing
\begin{align}
&\int_{\Rd}\int_{\Act} c_{n_k}(x,\zeta)v_{n_k}^*(x)(\D \zeta)\eta_{v_{n_k}^*}^{n_k}(\D x)\, -\int_{\Rd}\int_{\Act} c(x,\zeta)\tilde{v}^*(x)(\D \zeta)\eta_{\tilde{v}^*}(\D x) \nonumber \\
& =  \bigg(\int_{\Rd}\int_{\Act} c_{n_k}(x,\zeta)v_{n_k}^*(x)(\D \zeta)\eta_{v_{n_k}^*}^{n_k}(\D x) - \int_{\Rd}\int_{\Act} c_{n_k}(x,\zeta)v_{n_k}^*(x)(\D \zeta)\eta_{\tilde{v}^*}(\D x) \bigg) \nonumber \\
&\quad + \bigg(\int_{\Rd}\int_{\Act} c_{n_k}(x,\zeta)v_{n_k}^*(x)(\D \zeta)\eta_{\tilde{v}^*}(\D x) -\int_{\Rd}\int_{\Act} c(x,\zeta)\tilde{v}^*(x)(\D \zeta)\eta_{\tilde{v}^*}(\D x) \bigg)
\end{align}
and noting that the first term converges to zero by total variation convergence of $\eta_{v_{n_k}^*}^{n_k} \to \eta_{\tilde{v}^*}$  and the second term converging by the convergence in the control topology on $\Usm$ as $\eta_{\tilde{v}^*}$ is fixed; in view of the fact that $c_n \to c$ (continuously over control actions)  we conclude that $\Tilde{\rho} = \int_{\Rd}\int_{\Act} c(x,\zeta)\tilde{v}^*(x)(\D \zeta)\eta_{\tilde{v}^*}(\D x)$\,. Therefore, the pair $(\Tilde{V}, \Tilde{\rho})\in \Sobl^{2,p}(\Rd)\times \RR$, \, $1< p < \infty$, which has the properties that $\Tilde{V}(0) = 0$ and $\inf_{\Rd} \Tilde{V} > -\infty$, is a compatible solution (see \cite[Definition~1.1]{AA13}) to \cref{EErgonearOpt1A}. Since solution to the equation \cref{EErgonearOpt1A} is unique (see \cite[Theorem~1.1]{AA13}), it follows that $(\Tilde{V}, \Tilde{\rho}) \equiv (V, \rho)$\,. This completes the proof of the theorem.
\end{proof} 
In the following theorem, we prove existence and uniqueness of solution of a certain Poisson's equation. This will be useful in proving the robustness result.
\begin{theorem}\label{NearmonotPoisso}
Suppose that Assumptions (A1) - (A4) hold. Let $v\in\Usm$ be a stable control such that 
\begin{equation}\label{ENearmonotPoisso1}
\liminf_{\norm{x}\to\infty}\inf_{\zeta\in \Act} c(x,\zeta) > \inf_{x\in\Rd}\sE_x(c, v)\,. 
\end{equation} Then, there exists a unique pair $(V^v, \rho_v)\in \Sobl^{2,p}(\Rd)\times \RR$, \, $1< p < \infty$, with $V^v(0) = 0$ and $\inf_{\Rd} V^v > -\infty$ and $\rho_v \leq \int_{\Rd}\int_{\Act} c(x,\zeta)v(x)(\D \zeta)\eta_{v}(\D x)$, satisfying
\begin{equation}\label{EErgonearPoisso1A}
\rho_v = \left[\sL_{v}V^v(x) + c(x, v(x))\right]
\end{equation}
Moreover, we have
\begin{itemize}
\item[(i)]$\rho_v = \inf_{\Rd}\sE_x(c, v)$\,.
\item[(ii)] for all $x\in \Rd$
\begin{equation}\label{EErgonearPoisso1B}
V^v(x) \,=\, \lim_{r\downarrow 0}\Exp_{x}^v\left[\int_{0}^{\uuptau_{r}} \left( c(X_t, v(X_t)) - \rho_v\right)\D t\right]\,.
\end{equation}
\end{itemize} 
\end{theorem}
\begin{proof}
Since $c$ is bounded, we have $\left(\rho^{v} \,\df\,\right) \int_{\Rd}\int_{\Act} c(x,\zeta)v(x)(\D \zeta)\eta_{v}(\D x)\leq \inf_{\Rd}\sE_x(c, v) \leq \norm{c}_{\infty}$\,. Also, since (see, \cref{ENearmonotPoisso1}) $\liminf_{\norm{x}\to\infty}\inf_{\zeta\in \Act} c(x,\zeta) > \rho^{v}$, from \cite[Lemma~3.6.1]{ABG-book}, it follows that 
\begin{equation}\label{EErgonearPoisso1C}
\inf_{\kappa(\rho^v)}\cJ_{\alpha}^{v}(x, c) = \inf_{\Rd}\cJ_{\alpha}^{v}(x, c) \leq \frac{\rho^{v}}{\alpha}\,,
\end{equation} where $\kappa(\rho^v) \,\df\, \{x\in \Rd\mid \min_{\zeta\in\Act}c(x,\zeta) \leq \rho^v\}$ and $\cJ_{\alpha}^{v}(x, c)$ is the $\alpha$-discounted cost defined as in \cref{EDiscost}. It known that $\cJ_{\alpha}^{v}(x, c)$ is a solution to the Poisson's equation (see, \cite[Lemma~A.3.7]{ABG-book})
\begin{equation}\label{EErgonearPoisso1D}
 \sL_{v}\cJ_{\alpha}^{v}(x, c) - \alpha \cJ_{\alpha}^{v}(x, c) = - c(x, v(x))\,.
\end{equation} Since $\kappa(\rho^{v})$ is compact, for some $R_0>0$, we have $\kappa(\rho^v)\subset \sB_{R_{0}}$\,. Thus from \cite[Lemma~3.6.3]{ABG-book}, we deduce that for each $R> R_0$ there exist constants $\Tilde{C}_{2}(R), \Tilde{C}_{2}(R, p)$ depending only on $d, R_0$ such that
\begin{equation}\label{EErgonearPoisso1E}
\osc_{\sB_{2R}} \cJ_{\alpha}^{v}(x, c) \leq \Tilde{C}_{2}(R)\left(1 + \alpha\inf_{\sB_{R_0}}\cJ_{\alpha}^{v}(x, c) \right)\,,
\end{equation}
\begin{equation}\label{EErgonearPoisso1F}
\norm{\cJ_{\alpha}^{v}(\cdot, c) - \cJ_{\alpha}^{v}(0, c)}_{\Sob^{2,p}(\sB_R)}\leq \Tilde{C}_{2}(R, p) \left(1 + \alpha\inf_{\sB_{R_0}}\cJ_{\alpha}^{v}(x, c) \right)\,.
\end{equation}
Thus, arguing as in \cite[Lemma~3.6.6]{ABG-book}, we deduce that there exists $(V^{v}, \Tilde{\rho}_v)\in \Sobl^{2,p}(\Rd)\times \RR$ such that as $\alpha\to 0$, $\cJ_{\alpha}^{v}(\cdot, c) - \cJ_{\alpha}^{v}(0, c) \to V^{v}(\cdot)$ and $\alpha\cJ_{\alpha}^{v}(0, c) \to \Tilde{\rho}_{v}$ and the pair $(V^{v}, \Tilde{\rho}_v)$ satisfies
\begin{equation}\label{EErgonearPoisso1G}
 \sL_{v}V^{v}(x) + c(x, v(x)) = \Tilde{\rho}_{v}\,.
\end{equation} By \cref{EErgonearPoisso1C}, we get $\Tilde{\rho}_{v} \leq \rho^{v}$. Now, in view of estimates \cref{EErgonearPoisso1C} and \cref{EErgonearPoisso1F}, it is easy to see that
\begin{equation}\label{EErgonearPoisso1H}
\norm{V^{v}}_{\Sob^{2,p}(\sB_R)}\leq \Tilde{C}_{2}(R, p) \left(1 + M \right)\,.
\end{equation} Also, arguing as in Theorem~\ref{ErgoContnuity} (see \cref{EErgoContnuity1IA}), from estimate \cref{EErgonearPoisso1E} it follows that
\begin{equation}\label{EErgonearPoisso1I}
V^{v}\geq -\Tilde{C}_{2}(R_0) \left(1 + M \right)\,. 
\end{equation}
Now, applying It$\hat{\rm o}$-Krylov formula and using \cref{EErgonearPoisso1G} we obtain
\begin{align*}
 \Exp_x^{v}\left[V^{v}\left(X_{T\wedge \uptau_{R}}\right)\right] - V^v(x)\,=\, \Exp_x^{v}\left[\int_0^{T\wedge \uptau_{R}} \left(\Tilde{\rho}_{v} - c(X_t, v(X_t))\right) \D{t}\right]\,.
\end{align*} This implies
\begin{align*}
 \inf_{y\in\Rd}V^{v}(y) - V^v(x)\,\leq\, \Exp_x^{v}\left[\int_0^{T\wedge \uptau_{R}} \left(\Tilde{\rho}_{v} - c(X_t, v(X_t))\right) \D{t}\right]\,.
\end{align*}Since $v$ is stable, letting $R\to \infty$, we get
\begin{align*}
 \inf_{y\in\Rd}V^{v}(y) - V^v(x)\,\leq\, \Exp_x^{v}\left[\int_0^{T} \left(\Tilde{\rho}_{v} - c(X_t, v(X_t))\right) \D{t}\right]\,.
\end{align*}Now dividing both sides of the above inequality by $T$ and letting $T\to \infty$, it follows that
\begin{align*}
 \limsup_{T\to \infty}\frac{1}{T}\Exp_x^{v}\left[\int_0^{T} c(X_t, v(X_t)) \D{t}\right] \,\leq\, \Tilde{\rho}_{v}\,.
\end{align*} Thus, $\rho^v \leq \Tilde{\rho}_{v}$. This indeed implies that $\rho^v = \Tilde{\rho}_{v}$\,. The representation \cref{EErgonearPoisso1B} of $V^v$ follows by closely mimicking the argument of \cite[Lemma~3.6.9]{ABG-book}. Therefore, we have a solution pair $(V^v, \rho_v)$ to \cref{EErgonearPoisso1A} satisfying (i) and (ii). 

Next we want to prove that the solution pair is unique. To this end, let $(\hat{V}^v, \hat{\rho}_v)\in \Sobl^{2,p}(\Rd)\times \RR$, \, $1< p < \infty$, with $\hat{V}^v(0) = 0$ and $\inf_{\Rd} \hat{V}^v > -\infty$ and $\hat{\rho}_v \leq \int_{\Rd}\int_{\Act} c(x,\zeta)v(x)(\D \zeta)\eta_{v}(\D x)$, satisfying
\begin{equation}\label{EErgonearPoisso1J}
\hat{\rho}_v = \left[\sL_{v}\hat{V}^v(x) + c(x, v(x))\right]
\end{equation}   
Applying It$\hat{\rm o}$-Krylov formula and using \cref{EErgonearPoisso1J} we obtain
\begin{align}\label{EErgonearPoisso1L}
\limsup_{T\to \infty}\frac{1}{T}\Exp_x^{v}\left[\int_0^{T} c(X_t, v(X_t)) \D{t}\right] \,\leq\, \hat{\rho}_{v}
\end{align} Since, $\hat{\rho}_v \leq \inf_{\Rd}\sE_x(c, v)$, from \cref{EErgonearPoisso1L} we obtain $\hat{\rho}^{v} = \rho_{v}$\,. Now, from \cref{EErgonearPoisso1G}, applying It$\hat{\rm o}$-Krylov formula, we deduce that
\begin{align}\label{EErgonearPoisso1N}
\hat{V}^v(x)\,=\, \Exp_x^{v}\left[\int_0^{\uuptau_{r}\wedge \uptau_{R}} \left(c(X_t, v(X_t)) - \hat{\rho}_{v}\right) \D{t} + \hat{V}^{v}\left(X_{\uuptau_{r}\wedge \uptau_{R}}\right)\right]\,.
\end{align} Since $v$ is stable and $\hat{V}^v$ is bounded from below, for all $x\in \Rd$ we have
\begin{equation*}
\liminf_{R\to\infty}\Exp_x^{v}\left[\hat{V}^{v}\left(X_{\uptau_{R}}\right)\Ind_{\{\uuptau_{r}\geq \uptau_{R}\}}\right]\geq 0\,.
\end{equation*}Hence, letting $R\to\infty$ by Fatou's lemma from \cref{EErgonearPoisso1N},  it follows that
\begin{align*}
\hat{V}^v(x)&\,\geq\, \Exp_x^{v}\left[\int_0^{\uuptau_{r}} \left(c(X_t, v(X_t)) - \hat{\rho}_{v}\right) \D{t} +\hat{V}^{v}\left(X_{\uuptau_{r}}\right)\right]\nonumber\\
&\,\geq\, \Exp_x^{v}\left[\int_0^{\uuptau_{r}} \left(c(X_t, v(X_t)) - \hat{\rho}_{v}\right) \D{t}\right] +\inf_{\sB_r}\hat{V}^{v}\,.
\end{align*}Since $\hat{V}^{v}(0) =0$, letting $r\to 0$, we obtain
\begin{align}\label{EErgonearPoisso1o}
\hat{V}^v(x)\,\geq\, \limsup_{r\downarrow 0}\Exp_x^{v}\left[\int_0^{\uuptau_{r}} \left(c(X_t, v(X_t)) - \hat{\rho}_{v}\right) \D{t} \right]\,.
\end{align}
From \cref{EErgonearPoisso1B} and \cref{EErgonearPoisso1o}, it is easy to see that $V^v - \hat{V}^v \leq 0$ in $\Rd$. On the other hand by \cref{EErgonearPoisso1A} and \cref{EErgonearPoisso1J} one has $\sL_{v}\left(V^v - \hat{V}^v\right)(x) = 0$ in $\Rd$. Hence, applying strong maximum principle \cite[Theorem~9.6]{GilTru}, one has $V^v = \hat{V}^v$. This proves uniqueness.
\end{proof}
Next we prove the robustness result, i.e., we prove that $\sE_{x}(c, v_{n}^*)\to \rho$ as $n\to \infty$, where $v_{n}^*$ is an optimal ergodic control of the approximated model (see,  Theorem~\ref{ergodicnearmono2})\,. In order to establish this result we will assume that $\widehat{\Theta}\df \{\eta_{v_n^*}: n\in \NN\}$ is tight, where $\eta_{v_n^*}$ is the unique invariant measure of \cref{E1.1} corresponding to $v_n^*$\,.
\begin{theorem}\label{ErgodNearmonoRobu1}
Suppose that Assumptions (A1) - (A6) hold. Also, assume that 
\begin{equation}\label{ENearmonotApro1}
\liminf_{\norm{x}\to\infty}\inf_{\zeta\in \Act} c(x,\zeta) > \inf_{x\in\Rd}\sup_{n\in\NN}\sE_x(c, v_n^*)\,,
\end{equation} and the sets $\widehat{\Theta}$ and $\Theta$ are tight. Then, we have
\begin{equation}\label{EErgoRobust1}
\lim_{n\to\infty} \inf_{x\in\Rd}\sE_{x}(c, v_{n}^*) = \sE^{*}(c)\,.
\end{equation}
\end{theorem}
\begin{proof} We shall follow a similar proof program as that of Theorem ~\ref{TC1.4}, under the discounted setup. Since $c$ is bounded, we have $\left(\rho_{v_{n}^*} \,\df\,\right) \inf_{x\in \Rd}\sE_{x}(c, v_{n}^*) \leq \norm{c}_{\infty}$\,. From our assumption \cref{ENearmonotApro1}, we know that $\liminf_{\norm{x}\to\infty}\inf_{\zeta\in \Act} c(x,\zeta) > \rho_{v_{n}^*}$\,. Hence, from Theorem~\ref{NearmonotPoisso}, we have there exists a unique pair $(V^{v_{n}^*}, \rho_{v_{n}^*})\in \Sobl^{2,p}(\Rd)\times \RR$, \, $1< p < \infty$, with $V^{v_{n}^*}(0) = 0$ and $\inf_{\Rd} V^{v_{n}^*} > -\infty$, satisfying
\begin{equation}\label{ErgodNearmonoRobu1A}
\rho_{v_{n}^*} = \left[\sL_{v_{n}^*}V^{v_{n}^*}(x) + c(x, {v_{n}^*}(x))\right],
\end{equation} with $\rho_{v_{n}^*} = \inf_{x\in \Rd}\sE_{x}(c, v_{n}^*) = \int_{\Rd}\int_{\Act} c(x,\zeta)v_{n}^*(x)(\D \zeta)\eta_{v_{n}^*}(\D x)$\,.
Moreover, in view of assumption \cref{ENearmonotApro1}, from \cref{EErgonearPoisso1H} and \cref{EErgonearPoisso1I}, we have 
\begin{equation}\label{ErgodNearmonoRobu1B}
\norm{V^{v_{n}^*}}_{\Sob^{2,p}(\sB_R)}\leq \kappa_1\quad\text{and}\quad V^{v_{n}^*}(x)\geq - \kappa_2\,\,\, \text{for all}\,\,\, x\in \Rd\,,
\end{equation} where $\kappa_1, \kappa_2$ are constants independent of $n\in\NN$\,. Thus by the Banach-Alaoglu theorem and standard diagonalization argument (as in \cref{ETC1.3BC}), we deduce that exists $\hat{V}\in \Sobl^{2,p}(\Rd)$ such that along a sub-sequence
\begin{equation}\label{ErgodNearmonoRobu1C}
\begin{cases}
V^{v_{n_k}^*}\to & \hat{V}\quad \text{in}\quad \Sobl^{2,p}(\Rd)\quad\text{(weakly)}\\
V^{v_{n_k}^*}\to & \hat{V}\quad \text{in}\quad \cC^{1, \beta}_{loc}(\Rd)\quad\text{(strongly)}\,.
\end{cases}       
\end{equation}
Again, since $\rho_{v_{n}^*} \leq M$, along a further sub-sequence (without loss of generality denoting by same sequence), we have $\rho_{v_{n_k}^*}\to \hat{\rho}$ as $k\to \infty$\,. Since $\Usm$ is compact along a further subsequence (without loss of generality denoting by same sequence) we have $v_{n_k}^* \to \hat{v}^*$ as $k\to\infty$\,. Now, as before, multiplying by test function and letting $k\to\infty$, from \cref{ErgodNearmonoRobu1A}, we deduce that the pair $(\hat{V}, \hat{\rho})\in \Sobl^{2,p}(\Rd)\times \RR$, \, $1< p < \infty$, satisfies
\begin{equation}\label{ErgodNearmonoRobu1D}
\hat{\rho} = \left[\sL_{\hat{v}^*}\hat{V}(x) + c(x, {\hat{v}^*}(x))\right]
\end{equation} Since $V^{v_{n_k}^*}(0) = 0$ for all $k\in \NN$, it is easy to see that $\hat{V}(0) = 0$\,. Also, by \cref{ErgodNearmonoRobu1B}, it follows that 
$\inf_{\Rd} \hat{V} > -\infty$\,. Hence, using \cref{ENearmonotApro1} and \cref{ErgodNearmonoRobu1D}, we have $\hat{v}^*\in \Usm$ is stable. Since $\widehat{\Theta}$ is tight, in view of \cite[Lemma~3.2.6]{ABG-book}, it is easy to see that $\hat{\rho} = \int_{\Rd}\int_{\Act} c(x,\zeta)\hat{v}^*(x)(\D \zeta)\eta_{\hat{v}^*}(\D x)$\,. Thus, by Lemma~\ref{NearmonotPoisso}, we deduce that $(\hat{V}, \hat{\rho})\equiv (V^{\hat{v}^*}, \rho_{\hat{v}^*})$\,.

Note that 
\begin{equation*}
|\rho_{v_{n_k}^*} - \rho| \leq |\rho_{v_{n_k}^*} - \rho_{n_k}| + |\rho_{n_k} - \rho|\,. 
\end{equation*} Since $|\rho_{n_k} - \rho| \to 0$ as $k\to\infty $ (see, Theorem~\ref{ErgoContnuity}), to complete the proof we have to show that $|\rho_{v_{n_k}^*} - \rho_{n_k}|\to 0$ as $k\to \infty$\,. From Theorem~\ref{ergodicnearmono2}, we know that the pair $(V_{n_k}, \rho_{n_k})\in \Sobl^{2,p}(\Rd)\times \RR$, \, $1< p < \infty$, with $V_{n_k}(0) = 0$, satisfies
\begin{equation}\label{ErgodNearmonoRobu1E}
\rho_{n_k} = \min_{\zeta \in\Act}\left[\sL_{\zeta}^{n_k}V_{n_k}(x) + c_{n_k}(x, \zeta)\right]\,.
\end{equation} For any minimizing selector $v_{n_k}^*\in \Usm$, rewriting \cref{ErgodNearmonoRobu1E}, we get 
\begin{equation}\label{ErgodNearmonoRobu1F}
\rho_{n_k} = \left[\sL_{v_{n_k}^*}^{n_k}V_{n_k}(x) + c_{n_k}(x, v_{n_k}^*(x))\right]\,.
\end{equation} Now, in view of estimates \cref{EErgoContnuity1F} and \cref{EErgoContnuity1J}, it follows that
\begin{equation}\label{ErgodNearmonoRobu1G}
\norm{V_{n_k}}_{\Sob^{2,p}(\sB_R)}\leq \kappa_3\quad\text{and}\quad V_{n_k}(x)\geq - \kappa_4\,\,\, \text{for all}\,\,\, x\in \Rd\,,
\end{equation} where $\kappa_3, \kappa_4$ are constants independent of $k\in\NN$\,. Hence, by the Banach-Alaoglu theorem and standard diagonalization argument (see \cref{ETC1.3BC}), we have there exists $\bar{V}\in \Sobl^{2,p}(\Rd)$ such that along a sub-sequence
\begin{equation}\label{ErgodNearmonoRobu1H}
\begin{cases}
V_{n_k}\to & \bar{V}\quad \text{in}\quad \Sobl^{2,p}(\Rd)\quad\text{(weakly)}\\
V_{n_k}\to & \bar{V}\quad \text{in}\quad \cC^{1, \beta}_{loc}(\Rd) \quad\text{(strongly)}\,.
\end{cases}       
\end{equation}Also, $\rho_{n_k} \leq M$ implies that along a further subsequence (denoting by same sequence without loss generality) $\rho_{n_k} \to \bar{\rho}$. Since $v_{n_k}^* \to \hat{v}^*$ in $\Usm$, multiplying by test functions and letting $k\to\infty$ from \cref{ErgodNearmonoRobu1F}, we obtain that the pair $(\bar{V}, \bar{\rho})\in \Sobl^{2,p}(\Rd)\times \RR$, \, $1< p < \infty$ satisfies
\begin{equation}\label{ErgodNearmonoRobu1I}
\bar{\rho} = \left[\sL_{\hat{v}^*}\bar{V}(x) + c(x, \hat{v}^*(x))\right]\,.
\end{equation} Form \cref{ErgodNearmonoRobu1G}, it easy to see that $\inf_{\Rd}\bar{V} > -\infty$. Also, since $V_{n_k}(0) = 0$ for all $k\in\NN$, we have $\bar{V}(0) = 0$\,. Since $\Theta$ is tight, arguing as in proof of Theorem~\ref{ErgoContnuity}, we deduce that $\bar{\rho} = \int_{\Rd}\int_{\Act} c(x,\zeta)\hat{v}^*(x)(\D \zeta)\eta_{\hat{v}^*}(\D x)$\,. Thus, by uniqueness of solution of \cref{ErgodNearmonoRobu1I} (see, Theorem~\ref{NearmonotPoisso}) it follows that $(\bar{V}, \bar{\rho})\equiv (V^{\hat{v}^*}, \rho_{\hat{v}^*})$\,. Since both $\rho_{v_{n_k}^*}$ and $\rho_{n_k}$ converges to same limit $\rho_{\hat{v}^*}$, we deduce that  $|\rho_{v_{n_k}^*} - \rho_{n_k}|\to 0$ as $k\to \infty$\,. This completes the proof of the theorem.
\end{proof}
\subsection{Analysis under Lyapunov stability}\label{Lyapunov stability}
In this section we study the robustness problem for ergodic cost criterion under Lyapunov stability assumption. We assume the following Foster-Lyapunov condition on the dynamics.
\begin{itemize}
\item[\hypertarget{A7}{{(A7)}}]
\begin{itemize}
\item[(i)]
There exists a positive constant $\widehat{C}_0$, and a pair of inf-compact  functions $(\Lyap, h)\in \cC^{2}(\Rd)\times\cC(\Rd\times\Act)$ (i.e., the sub-level sets $\{\Lyap\leq k\} \,,\{h\leq k\}$ are compact or empty sets in $\Rd$\,, $\Rd\times\Act$ respectively for each $k\in\RR$) such that
\begin{equation}\label{Lyap1}
\sL_{\zeta}\Lyap(x) \leq \widehat{C}_{0} - h(x,u)\quad \text{for all}\,\,\, (x,\zeta)\in \Rd\times \Act\,,
\end{equation} where  $h$ is locally Lipschitz continuous in its first argument uniformly with respect to second. 
\item[(ii)] For each $n\in\NN$\,, we have
\begin{equation}\label{Lyap2}
\sL_{\zeta}^{n}\Lyap(x) \leq \widehat{C}_{0} - h(x,u)\quad \text{for all}\,\,\, (x,\zeta)\in \Rd\times \Act\,,
\end{equation} where the functions $\Lyap, h$ are as in \cref{Lyap1}\,. 
\end{itemize} 
\end{itemize} 
Combining \cite[Theorem~3.7.11]{ABG-book} and \cite[Theorem~3.7.12]{ABG-book}, we have the following complete characterization of the ergodic optimal control.
\begin{theorem}\label{TErgoOpt1}
Suppose that assumptions (A1)-(A4) and (A7)(i) hold. Then the ergodic HJB equation
\begin{equation}\label{EErgoOpt1A}
\rho = \min_{\zeta \in\Act}\left[\sL_{\zeta}V^*(x) + c(x, \zeta)\right]
\end{equation} admits unique solution $(V^*, \rho)\in \cC^2(\Rd)\cap \sorder(\Lyap)\times \RR$ satisfying $V^*(0) = 0$\,.
Moreover, we have
\begin{itemize}
\item[(i)]$\rho = \sE^*(c)$
\item[(ii)] a stationary Markov control $v^*\in \Usm$ is an optimal control (i.e., $\sE_x(c, v^*) = \sE^*(c)$) if and only if it satisfies
\begin{equation}\label{EErgoOpt1B}
\min_{\zeta \in\Act}\left[\sL_{\zeta}V^*(x) + c(x, \zeta)\right] \,=\, \trace\bigl(a(x)\grad^2 V^*(x)\bigr) + b(x,v^*(x))\cdot \grad V^*(x) + c(x, v^*(x))\,,\quad \text{a.e.}\,\, x\in\Rd\,.
\end{equation}
\item[(iii)] for any $v^*\in \Usm$ satisfying \cref{EErgoOpt1B}, we have
\begin{equation}\label{EErgoOpt1C}
V^*(x) \,=\, \lim_{r\downarrow 0}\Exp_{x}^{v^*}\left[\int_{0}^{\uuptau_{r}} \left( c(X_t, v^*(X_t)) - \sE^*(c)\right)\D t\right] \quad\text{for all}\,\,\, x\in \Rd\,.
\end{equation}
\end{itemize} 
\end{theorem} 
Again, from \cite[Theorem~3.7.11]{ABG-book} and \cite[Theorem~3.7.12]{ABG-book}, for the approximated model for each $n\in\NN$, we have the following complete characterization of the optimal control.
\begin{theorem}\label{TErgoOptApprox1}
Suppose that Assumptions (A5) and (A7)(ii) hold. Then the ergodic HJB equation
\begin{equation}\label{TErgoOptApprox1A}
\rho_n = \min_{\zeta \in\Act}\left[\sL_{\zeta}^nV(x) + c_n(x, \zeta)\right]
\end{equation} admits unique solution $(V^{n*}, \rho_n)\in \cC^2(\Rd)\cap \sorder(\Lyap)\times \RR$ satisfying $V^{n*}(0) = 0$\,.
Moreover, we have
\begin{itemize}
\item[(i)]$\rho_n = \sE^{n*}(c_n)$
\item[(ii)] a stationary Markov control $v_n^*\in \Usm$ is an optimal control (i.e., $\sE_x^n(c_n, v_n^{*}) = \sE^{n*}(c_n)$) if and only if it satisfies
\begin{equation}\label{TErgoOptApprox1B}
\min_{\zeta \in\Act}\left[\sL_{\zeta}^n V^{n*}(x) + c_n(x, \zeta)\right] \,=\, \trace\bigl(a_n(x)\grad^2 V^{n*}(x)\bigr) + b_n(x,v_n^*(x))\cdot \grad V^{n*}(x) + c(x, v_n^*(x))\,,\quad \text{a.e.}\,\, x\in\Rd\,.
\end{equation}
\item[(iii)] for any $v_n^*\in \Usm$ satisfying \cref{TErgoOptApprox1B}, we have
\begin{equation}\label{TErgoOptApprox1C}
V^{n*}(x) \,=\, \lim_{r\downarrow 0}\Exp_{x}^{v_n^*}\left[\int_{0}^{\uuptau_{r}} \left( c_n(X_t, v_n^*(X_t)) - \sE^{n*}(c_n)\right)\D t\right] \quad\text{for all}\,\,\, x\in \Rd\,.
\end{equation}
\end{itemize} 
\end{theorem}
From \cite[lemma~3.7.8]{ABG-book}, it is easy to see that the functions $V^{n*}, V^{*}$ are bounded from below. Next we show that under Assumption (A7), as $n\to\infty$ the optimal value $V^{n*}$ of the approximated model converges to the optimal value $V^{*}$ of the true model.
\begin{theorem}\label{TErgoOptCont}
Suppose that Assumptions (A1)-(A5) and (A7) hold. Then, it follows that 
\begin{equation}\label{ETErgoOptCont1}
\lim_{n\to\infty} \sE^{n*}(c_n) = \sE^{*}(c)\,.
\end{equation}
\end{theorem}
\begin{proof}
Since, $\|c_n\|_{\infty} \leq M,$ we get $\sE^{n*}(c_n) \leq M$\,. Also, \cref{Lyap1} implies that all $v\in\Usm$ is stable and $\inf_{v\in\Usm}\eta_v(\sB_R) > 0$ for any $R>0$ (see, \cite[Lemma~3.3.4]{ABG-book} and \cite[Lemma~3.2.4(b)]{ABG-book}). Thus from \cite[Theorem~3.7.6]{ABG-book}, we have there exist constants $\widehat{C}_1, \widehat{C}_2 >0$ depending only on the radius $R>0$ such that for all $\alpha >0$, we have
\begin{equation}\label{ETErgoOptCont1A}
\|V_\alpha^n(\cdot) - V_\alpha^n(0)\|_{\Sob^{2,p}(\sB_{R})} \leq \widehat{C}_1 \quad \text{and}\,\,\, \sup_{\sB_R}\alpha V_{\alpha}^{n} \leq \widehat{C}_2\,.
\end{equation}By standard vanishing discount argument (see \cite[Lemma~3.7.8]{ABG-book}) as $\alpha\to 0$ we have $V_\alpha^n(\cdot) - V_\alpha^n(0) \to V^{n*}$ and $\alpha V_\alpha^n(0) \to \rho_n$\,. Hence the estimates \cref{ETErgoOptCont1A} give us $\|V^{n*}\|_{\Sob^{2,p}(\sB_{R})} \leq \widehat{C}_1$\,. Since the constant $\widehat{C}_1$ is independent of $n$, by standard diagonalization argument and the Banach-Alaoglu theorem, we can extract a subsequence $\{V^{n_k*}\}$ such that for some $\widehat{V}^*\in \Sobl^{2,p}(\Rd)$ (as in \cref{ETC1.3BC})
\begin{equation}\label{ETErgoOptCont1B}
\begin{cases}
V^{n_k*}\to & \widehat{V}^*\quad \text{in}\quad \Sobl^{2,p}(\Rd)\quad\text{(weakly)}\\
V^{n_k*}\to & \widehat{V}^*\quad \text{in}\quad \cC^{1, \beta}_{loc}(\Rd) \quad\text{(strongly)}\,.
\end{cases} 
\end{equation} Also, since $\rho_n \leq M$ , along a further sub-sequence (with out loss of generality denoting by same sequence) we have $\rho_{n_k}\to \widehat{\rho}$ as $k\to \infty$\,. Now multiplying both sides of the equation \cref{TErgoOptApprox1A} by test functions $\phi$, we obtain 
\begin{equation*}
\int_{\Rd}\trace\bigl(a_{n_k}(x)\grad^2 V^{n_k*}(x)\bigr)\phi(x)\D x + \int_{\Rd} \min_{\zeta\in \Act} \{b_{n_k}(x,\zeta)\cdot \grad V^{n_k*}(x) + c_{n_k}(x, \zeta)\}\phi(x)\D x = \int_{\Rd} \rho_{n_k}\phi(x)\D x\,.
\end{equation*}
As in Theorem~\ref{TC1.4}, using \cref{ETErgoOptCont1B} and letting $k\to\infty$ it follows that $\widehat{V}^*\in \Sobl^{2,p}(\Rd)$ satisfies
\begin{equation}\label{TErgoOptCont1C}
\widehat{\rho} = \min_{\zeta \in\Act}\left[\sL_{\zeta}\widehat{V}^*(x) + c(x, \zeta)\right]\,.
\end{equation}Rewriting the equation \cref{TErgoOptCont1C}, we have 
\begin{equation*}
\trace\bigl(a(x)\grad^2 \widehat{V}^*(x)\bigr) =  f(x)\,,\quad \text{a.e.}\,\, x\in\Rd\,,
\end{equation*} where 
\begin{equation*}
f(x) = - \inf_{\zeta\in \Act}\left[b(x,\zeta)\cdot \grad \widehat{V}^*(x) + c(x, \zeta) - \widehat{\rho} \right]\,.
\end{equation*} In view of \cref{ETErgoOptCont1B} and assumptions (A1) and (A2), it is easy to see that $f\in \cC^{0, \beta}_{loc}(\Rd)$ where $0< \beta < 1 - \frac{d}{p}$\,. Thus, by elliptic regularity \cite[Theorem~3]{CL89} (also see, \cite[Theorem~9.19]{GilTru}), we obtain $\widehat{V}^*\in \cC^{2}(\Rd)$\,.  

Next we want to show that $\widehat{V}^*\in \sorder{(\Lyap)}$. Since $\sup_{n}\|c_n\| \leq M$ we have $1 + \tilde{c} \in \sorder{(h)}$, where $\tilde{c} \,\df\, \sup_n c_n$\,. Also, since $h$ is inf-compact for a large enough $r > 0$ we have $\displaystyle{\widehat{C}_0 - \inf_{\zeta\in\Act} h(x, \zeta) \leq - \epsilon}$ for all $x\in \sB_r^c$\,. Let $X_t^n$ be the solution of \cref{ASE1.1} corresponding to $v\in \Usm$\,. Hence, in view of \cref{Lyap2}, by It$\hat{\rm o}$-Krylov formula, for any $v\in \Usm$ and $x\in \sB_r^c\cap\sB_R$ we deduce that
\begin{equation*}
\Exp_{x}^{v}\left[\Lyap(X_{\uuptau_{r}^n \wedge \uptau_{R}^n}^n)\right] - \Lyap(x) = \Exp_{x}^{v}\left[\int_{0}^{\uuptau_{r}^n \wedge \uptau_{R}^n} \sL_{v}^n \Lyap(X_s^n) \D s\right] \leq -\epsilon \Exp_{x}^{v}\left[\uuptau_{r}^n \wedge \uptau_{R}^n\right]\,,  
\end{equation*} where $\uuptau_{r}^n \df \inf\{t\geq 0: X_t^n\in \sB_r\}$ and $\uptau_{R}^n \df \inf \{t\geq 0: X_t^n\in \sB_R^c\}$\,. Letting $R\to \infty$, by Fatou's lemma we obtain
\begin{equation}\label{TErgoOptCont1D}
\Exp_{x}^{v}\left[\uuptau_{r}^n\right] \leq \frac{1}{\epsilon} \Lyap(x)\quad \text{for all}\,\,\, x\in \sB_r^c\,\,\,\text{and}\,\,\, n\in\NN\,. 
\end{equation}
Again, by It$\hat{\rm o}$-Krylov formula, for any $v\in \Usm$ and $x\in \sB_r^c\cap\sB_R$ we have
\begin{equation*}
\Exp_{x}^{v}\left[\Lyap(X_{\uuptau_{r}^n \wedge \uptau_{R}^n}^n)\right] - \Lyap(x) = \Exp_{x}^{v}\left[\int_{0}^{\uuptau_{r}^n \wedge \uptau_{R}^n} \sL_{v}^n \Lyap(X_s^n) \D s\right] \leq  \Exp_{x}^{v}\left[\int_0^{\uuptau_{r}^n \wedge \uptau_{R}^n} (\widehat{C}_0 - h(X_s^n, v(X_s^n)))\D s \right]\,,  
\end{equation*} Thus, by Fatou's lemma letting $R\to\infty$ and using \cref{TErgoOptCont1D} we get
\begin{equation*}
\sup_{n\in\NN}\sup_{v\in\Usm}\Exp_{x}^{v}\left[\int_0^{\uuptau_{r}^n} h(X_s^n, v(X_s^n)) \D s\right] \leq \widehat{M}_1 \Lyap(x)\,,  
\end{equation*} for some positive constant $\widehat{M}_1$\,. Hence, by arguing as in the proof of \cite[Lemma~3.7.2 (i)]{ABG-book}, we have
\begin{equation}\label{TErgoOptCont1E}
\sup_{n\in \NN}\sup_{v\in\Usm} \Exp_{x}^{v}\left[\int_{0}^{\uuptau_{r}^n}(1 + \tilde{c}(X_s^n, v(X_s^n)))\D s\right]\in \sorder{(\Lyap)}\,.
\end{equation} Now, following the proof of \cite[Lemma~3.7.8]{ABG-book} (see, eq.(3.7.47)), it follows that
\begin{equation}\label{TErgoOptCont1F}
V^{n*}(x) \,\leq\, \sup_{v\in\Usm}\Exp_{x}^{v}\left[\int_{0}^{\uuptau_{r}^n} \left( c_n(X_t^n, v(X_t^n)) - \sE^{n*}(c_n)\right)\D t + V^{n*}(X_{\uuptau_{r}})\right]\,.
\end{equation} We know that for $p \geq d+1$ the space $\Sob^{2,p}(\sB_{R})$ is compactly embedded in $\cC^{1, \beta}(\bar{\sB}_R)$ where $0< \beta < 1 - \frac{d}{p}$\,. Since $\|V^{n*}\|_{\Sob^{2,p}(\sB_{R})} \leq \widehat{C}_1$ for some positive constant $\widehat{C}_1$ which depends only on $R$, we deduce that $\sup_{n\in\NN}\sup_{\sB_r}|V^{n*}| \leq \widehat{M}_2$, where $\widehat{M}_2 (>0)$ is a constant. Also, since $\sE^{n*}(c_n) \leq \|c_n\|_{\infty} \leq M$  from \cref{TErgoOptCont1F}, it is easy to see that
\begin{equation}\label{TErgoOptCont1H}
|V^{n*}(x)| \,\leq\, M\sup_{n\in\NN}\sup_{v\in\Usm}\Exp_{x}^{v}\left[\int_{0}^{\uuptau_{r}^n} \left( \tilde{c}(X_t^n, v(X_t^n)) + 1\right)\D t + \sup_{n\in\NN}\sup_{\sB_r}|V^{n*}|\right]\,.
\end{equation}
Therefore, by combining \cref{ETErgoOptCont1B}, \cref{TErgoOptCont1E} and \cref{TErgoOptCont1H}, we obtain $\widehat{V}^*\in \sorder{(\Lyap)}$\,. Since, $(\widehat{V}^*, \widehat{\rho})\in \cC^2(\Rd)\cap \sorder(\Lyap)\times \RR$ satisfying $V^*(0) = 0$ satisfies \cref{EErgoOpt1A}, by uniqueness result of Theorem~\ref{TErgoOpt1}, we deduce that $(\widehat{V}^*, \widehat{\rho}) \equiv (V^*, \rho)$. This completes the proof of the theorem.      
\end{proof}
Next theorem proofs existence of a unique solution to a certain equation in some suitable function space. This result will be very useful in establishing our robustness result.
\begin{theorem}\label{TErgoExisPoiss1}
Suppose that assumptions (A1)-(A4) and (A7)(i) hold. Then for each $v\in \Usm$ there exist a unique solution pair $(V^v, \rho^{v})\in \Sobl^{2,p}(\Rd)\cap \sorder(\Lyap)\times \RR$ for any $p >1$ satisfying
\begin{equation}\label{TErgoExisPoiss1A}
\rho^{v} = \sL_{v}V^v(x) + c(x, v(x))\quad\text{with}\quad V^v(0) = 0\,.
\end{equation}
Furthermore, we have
\begin{itemize}
\item[(i)]$\rho^{v} = \sE_x(c, v)$
\item[(ii)] for all $x\in\Rd$, we have
\begin{equation}\label{TErgoExisPoiss1B}
V^v(x) \,=\, \lim_{r\downarrow 0}\Exp_{x}^{v}\left[\int_{0}^{\uuptau_{r}} \left( c(X_t, v(X_t)) - \sE_x(c, v)\right)\D t\right]\,.
\end{equation}
\end{itemize} 
\end{theorem} 
\begin{proof} Existence of a solution pair $(V^v, \rho^{v})\in \Sobl^{2,p}(\Rd)\cap \sorder(\Lyap)\times \RR$ for any $p >1$ satisfying (i) and (ii) follows from \cite[Lemma~3.7.8]{ABG-book}\,. Now we want to prove the uniqueness of the solutions of \cref{TErgoExisPoiss1A}. Let $(\bar{V}^v, \bar{\rho}^{v})\in \Sobl^{2,p}(\Rd)\cap \sorder(\Lyap)\times \RR$ for any $p >1$ be any other solution pair of \cref{TErgoExisPoiss1A} with $\bar{V}^v(0) = 0$. By It$\hat{\rm o}$-Krylov formula, for $R>0$ we obtain
\begin{align}\label{TErgoExisPoiss1C}
\Exp_{x}^{v}\left[\bar{V}^v(X_{T\wedge\uptau_{R}})\right] - \bar{V}^v(x) &= \Exp_{x}^{v}\left[\int_{0}^{T\wedge\uptau_{R}} \sL_{v} \bar{V}^v(X_s) \D s\right]\nonumber\\
& = \Exp_{x}^{v}\left[\int_{0}^{T\wedge\uptau_{R}} \left(\bar{\rho}^{v} - c(X_s, v(X_s))\right)\D s \right]\,.  
\end{align}
Note that 
\begin{equation*}
\int_{0}^{T\wedge\uptau_{R}} \left(\bar{\rho}^{v} - c(X_s, v(X_s))\right)\D s = \int_{0}^{T\wedge\uptau_{R}} \bar{\rho}^{v}\D s - \int_{0}^{T\wedge\uptau_{R}}c(X_s, v(X_s))\D s
\end{equation*} Thus, letting $R\to \infty$ by monotone convergence theorem, we get
\begin{equation*}
\lim_{R\to\infty}\Exp_{x}^{v}\left[\int_{0}^{T\wedge\uptau_{R}} \left(\bar{\rho}^{v} - c(X_s, v(X_s))\right)\D s \right] = \Exp_{x}^{v}\left[\int_{0}^{T} \left(\bar{\rho}^{v} - c(X_s, v(X_s))\right)\D s \right]\,.
\end{equation*} Since $\bar{V}^v \in \sorder{(\Lyap)}$, in view of \cite[Lemma~3.7.2 (ii)]{ABG-book}, letting $R\to\infty$, we deduce that
\begin{align}\label{TErgoExisPoiss1D}
\Exp_{x}^{v}\left[\bar{V}^v(X_{T})\right] - \bar{V}^v(x) = \Exp_{x}^{v}\left[\int_{0}^{T} \left(\bar{\rho}^{v} - c(X_s, v(X_s))\right)\D s \right]\,.  
\end{align} Also, from \cite[Lemma~3.7.2 (ii)]{ABG-book}, we have
\begin{equation*}
\lim_{T\to\infty}\frac{\Exp_{x}^{v}\left[\bar{V}^v(X_{T})\right]}{T} = 0\,. 
\end{equation*}
Now, dividing both sides of \cref{TErgoExisPoiss1D} by $T$ and letting $T\to\infty$, we obtain 
\begin{align*}
\bar{\rho}^{v} = \limsup_{T\to \infty}\frac{1}{T}\Exp_{x}^{v}\left[\int_{0}^{T} \left(c(X_s, v(X_s))\right)\D s \right]\,.
\end{align*}This implies that $\bar{\rho}^{v} = \rho^{v}$\,. Using \cref{TErgoExisPoiss1A}, by It$\hat{\rm o}$-Krylov formula we have
\begin{align}\label{TErgoExisPoiss1E}
\bar{V}^v(x)\,=\, \Exp_x^{v}\left[\int_0^{\uuptau_{r}\wedge \uptau_{R}} \left(c(X_t, v(X_t)) - \bar{\rho}^{v}\right) \D{t} + \bar{V}^{v}\left(X_{\uuptau_{r}\wedge \uptau_{R}}\right)\right]\,.
\end{align} Also, by It$\hat{\rm o}$-Krylov formula and using \cref{Lyap1} it follows that
\begin{equation*}
\Exp_x^{v}\left[\Lyap\left(X_{\uptau_{R}}\right)\Ind_{\{\uuptau_{r}\geq \uptau_{R}\}}\right]\leq \widehat{C}_0 \Exp_x^{v}\left[\uuptau_{r}\right] + \Lyap(x)\quad \text{for all} \,\,\, r <|x|<R\,.
\end{equation*} Since $\bar{V}^v \in \sorder(\Lyap)$, form the above estimate, we get
\begin{equation*}
\liminf_{R\to\infty}\Exp_x^{v}\left[\bar{V}^{v}\left(X_{\uptau_{R}}\right)\Ind_{\{\uuptau_{r}\geq \uptau_{R}\}}\right] = 0\,.
\end{equation*}Thus, letting $R\to\infty$ by Fatou's lemma from \cref{TErgoExisPoiss1E}, it follows that
\begin{align*}
\bar{V}^v(x)&\,\geq\, \Exp_x^{v}\left[\int_0^{\uuptau_{r}} \left(c(X_t, v(X_t)) - \bar{\rho}^{v}\right) \D{t} +\bar{V}^{v}\left(X_{\uuptau_{r}}\right)\right]\nonumber\\
&\,\geq\, \Exp_x^{v}\left[\int_0^{\uuptau_{r}} \left(c(X_t, v(X_t)) - \bar{\rho}^{v}\right) \D{t}\right] +\inf_{\sB_r}\bar{V}^{v}\,.
\end{align*}Since $\bar{V}^{v}(0) =0$, letting $r\to 0$, we deduce that
\begin{align}\label{TErgoExisPoiss1F}
\bar{V}^v(x)\,\geq\, \limsup_{r\downarrow 0}\Exp_x^{v}\left[\int_0^{\uuptau_{r}} \left(c(X_t, v(X_t)) - \bar{\rho}^{v}\right) \D{t} \right]\,.
\end{align}
Since $\widehat{\rho}^v = \bar{\rho}^v$, from \cref{TErgoExisPoiss1B} and \cref{TErgoExisPoiss1F}, it is easy to see that $\widehat{V}^v - \bar{V}^v \leq 0$ in $\Rd$. Also, since $(V^v, \rho^v)$ and $(\bar{V}^v, \bar{\rho}^v)$ are two solution pairs of \cref{TErgoExisPoiss1A}, we have $\sL_{v}\left(\widehat{V}^v - \bar{V}^v\right)(x) = 0$ in $\Rd$. Hence, by strong maximum principle \cite[Theorem~9.6]{GilTru}, one has $V^v = \bar{V}^v$. This proves the uniqueness\,.
\end{proof}
Now we are ready to prove the robustness result, i.e., we want to show that $\sE_{x}(c, v_{n}^*)\to \sE^*(c)$ as $n\to \infty$, where $v_{n}^*$ is an optimal ergodic control of the approximated model (see, Theorem~\ref{TErgoOptApprox1})\,.
\begin{theorem}\label{ErgodLyapRobu1}
Suppose that Assumptions (A1) - (A5) and (A7) hold.  Then, we have
\begin{equation}\label{ErgodLyapRobu1A}
\lim_{n\to\infty} \inf_{x\in\Rd}\sE_{x}(c, v_{n}^*) = \sE^{*}(c)\,.
\end{equation}
\end{theorem}
\begin{proof} We shall follow a similar proof program as that of Theorem \ref{TC1.4}, under the discounted setup. From Theorem~\ref{TErgoExisPoiss1}, we know that for each $n\in \NN$ there exists a unique pair $(V^{v_{n}^*}, \rho^{v_{n}^*})\in \Sobl^{2,p}(\Rd)\cap\sorder{(\Lyap)}\times \RR$, \, $1< p < \infty$, with $V^{v_{n}^*}(0) = 0$ satisfying
\begin{equation}\label{ErgodLyapRobu1B}
\rho^{v_{n}^*} = \left[\sL_{v_{n}^*}V^{v_{n}^*}(x) + c(x, {v_{n}^*}(x))\right]
\end{equation}
In view of \cref{Lyap1}, it is easy to see that, each $v\in\Usm$ is stable and $\inf_{v\in\Usm}\eta_v(\sB_R) > 0$ for any $R>0$ (see, \cite[Lemma~3.3.4]{ABG-book} and \cite[Lemma~3.2.4(b)]{ABG-book}). Thus, from \cite[Theorem~3.7.4]{ABG-book}, it follows that $\norm{V^{v_{n}^*}}_{\Sob^{2,p}(\sB_R)}\leq \hat{\kappa}_1$ where $\hat{\kappa}_1$ is a constant independent of $n\in\NN$\,. Therefore by the Banach-Alaoglu theorem and standard diagonalization argument (as in \cref{ETC1.3BC}), we deduce that there exists $\tilde{V}\in \Sobl^{2,p}(\Rd)$ such that along a sub-sequence
\begin{equation}\label{ErgodLyapRobu1C}
\begin{cases}
V^{v_{n_k}^*}\to & \tilde{V}\quad \text{in}\quad \Sobl^{2,p}(\Rd)\quad\text{(weakly)}\\
V^{v_{n_k}^*}\to & \tilde{V}\quad \text{in}\quad \cC^{1, \beta}_{loc}(\Rd)\quad\text{(strongly)}\,.
\end{cases}       
\end{equation}
Again, since $\rho^{v_{n}^*} \leq M$, along a further sub-sequence (without loss of generality denoting by same sequence), we have $\rho^{v_{n_k}^*}\to \tilde{\rho}$ as $k\to \infty$\,. Since $\Usm$ is compact along a further subsequence (without loss of generality denoting by same sequence) we have $v_{n_k}^* \to \tilde{v}^*$ as $k\to\infty$\,. Now, as in Theorem~\ref{TC1.4}, multiplying by test function and letting $k\to\infty$, from \cref{ErgodLyapRobu1B}, it is easy to see that $(\tilde{V}, \tilde{\rho})\in \Sobl^{2,p}(\Rd)\times \RR$, \, $1< p < \infty$, satisfies
\begin{equation}\label{ErgodLyapRobu1D}
\tilde{\rho} = \left[\sL_{\tilde{v}^*}\tilde{V}(x) + c(x, {\tilde{v}^*}(x))\right]
\end{equation} As we know that $V^{v_{n_k}^*}(0) = 0$ for all $k\in \NN$, we deduce that $\tilde{V}(0) = 0$\,. Arguing as in Theorem~\ref{TErgoOptCont} and using the estimate $\norm{V^{v_{n}^*}}_{\Sob^{2,p}(\sB_R)}\leq \hat{\kappa}_1$, we have
\begin{equation}\label{ErgodLyapRobu1E}
|\tilde{V}(x)| \,\leq\, M\sup_{v\in\Usm}\Exp_{x}^{v}\left[\int_{0}^{\uuptau_{r}} \left( c(X_t, v(X_t)) + 1\right)\D t + \sup_{n\in\NN}\sup_{\sB_r}|V^{v_n^*}|\right] \in \sorder{(\Lyap)}\,.
\end{equation}
Thus, by uniqueness of solution of \cref{ErgodLyapRobu1D} (see, Theorem~\ref{TErgoExisPoiss1}), we deduce that $(\tilde{V}, \tilde{\rho})\equiv (V^{\tilde{v}^*}, \rho^{\tilde{v}^*})$\,.

By the triangle inequality 
\begin{equation*}
|\rho^{v_{n_k}^*} - \rho| \leq |\rho^{v_{n_k}^*} - \rho_{n_k}| + |\rho_{n_k} - \rho|\,. 
\end{equation*} From Theorem~\ref{TErgoOptApprox1} we have $|\rho_{n_k} - \rho| \to 0$ as $k\to\infty $. Hence to complete the proof we have to show that $|\rho^{v_{n_k}^*} - \rho_{n_k}|\to 0$ as $k\to \infty$\,. Now, for any minimizing selector $v_{n_k}^*\in \Usm$ of \cref{TErgoOptApprox1A}, we have 
\begin{equation}\label{ErgodLyapRobu1F}
\rho_{n_k} = \left[\sL_{v_{n_k}^*}^{n_k}V^{n_k}(x) + c_{n_k}(x, v_{n_k}^*(x))\right]\,.
\end{equation} In view of the estimate \cref{ETErgoOptCont1A}, we obtain
\begin{equation}\label{ErgodLyapRobu1G}
\norm{V^{n_k}}_{\Sob^{2,p}(\sB_R)}\leq \hat{\kappa}
\end{equation} where $\hat{\kappa} > 0$ is a constant independent of $k\in\NN$\,. Hence, by the Banach-Alaoglu theorem and standard diagonalization argument (see \cref{ETC1.3BC}), we have there exists $\tilde{V}^*\in \Sobl^{2,p}(\Rd)$ such that along a sub-sequence
\begin{equation}\label{ErgodLyapRobu1H}
\begin{cases}
V^{n_k}\to & \tilde{V}^*\quad \text{in}\quad \Sobl^{2,p}(\Rd)\quad\text{(weakly)}\\
V^{n_k}\to & \tilde{V}^*\quad \text{in}\quad \cC^{1, \beta}_{loc}(\Rd) \quad\text{(strongly)}\,.
\end{cases}       
\end{equation}Since, $\rho_{n_k} \leq M$ long a further subsequence (denoting by same sequence without loss generality) $\rho_{n_k} \to \tilde{\rho}^*$. As we know $v_{n_k}^* \to \tilde{v}^*$ in $\Usm$, multiplying both sides of \cref{ErgodLyapRobu1F} by test functions and letting $k\to\infty$, it follows that $(\tilde{V}^*, \tilde{\rho}^*)\in \Sobl^{2,p}(\Rd)\times \RR$, \, $1< p < \infty$ satisfies
\begin{equation}\label{ErgodLyapRobu1I}
\tilde{\rho}^* = \left[\sL_{\tilde{v}^*}\tilde{V}^*(x) + c(x, \tilde{v}^*(x))\right]\,.
\end{equation} Arguing as in Theorem~\ref{TErgoOptCont}, one can show that $\tilde{V}^*\in \sorder{(\Lyap)}$. Hence, by uniqueness of solution of \cref{ErgodLyapRobu1F} (see, Theorem~\ref{TErgoExisPoiss1}) we deduce that $(\tilde{V}^*, \tilde{\rho}^*)\equiv (V^{\tilde{v}^*}, \rho^{\tilde{v}^*})$\,. Since both $\rho^{v_{n_k}^*}$ and $\rho_{n_k}$ converge to same limit $\rho^{\tilde{v}^*}$, it follows that  $|\rho^{v_{n_k}^*} - \rho_{n_k}|\to 0$ as $k\to \infty$\,. This completes the proof of the theorem.
\end{proof}
\section{Finite Horizon Cost}\label{Finitecost}
In this section we study the robustness problem under a finite horizon criterion\,. We will assume that $a, a_n,  b, b_n, c, c_n$ satisfy the following:
\begin{itemize}
\item[\hypertarget{FN1}{{(FN1)}}]
The functions $a, a_n,  b, b_n, c, c_n$ satisfy
\begin{equation*}
\sup_{(x,\zeta)\in \Rd\times \Act}\left[\abs{b(x,\zeta)} + \norm{a(x)} + \sum_{i}^{d} \norm{\frac{\partial{a}}{\partial x_i}(x)} + \abs{c(x, \zeta)}\right] \,\le\, \mathrm{K}\,.
\end{equation*}
and 
\begin{equation*}
\sup_{n\in \NN}\sup_{(x,\zeta)\in \Rd\times \Act}\left[\abs{b_n(x,\zeta)} + \norm{a_n(x)} + \sum_{i}^{d} \norm{\frac{\partial{a_n}}{\partial x_i}(x)} + \abs{c_n(x, \zeta)}\right] \,\le\, \mathrm{K}\,.
\end{equation*} for some positive constant $\mathrm{K}$\,. Furthermore, $H\in \Sob^{2,p,\mu}(\Rd)\cap \Lp^{\infty}(\Rd)$\,,\,\, $p\ge 2$\,. 
\end{itemize} 

From \cite[Theorem~3.3, p. 235]{BL84-book}, the finite horizon optimality equation (or, the HJB equation)
\begin{align}\label{EFinitecost1A}
&\frac{\partial \psi}{\partial t} + \inf_{\zeta\in \Act}\left[\sL_{\zeta}\psi + c(x, \zeta) \right] = 0 \\
& \psi(T,x) = H(x)
\end{align} admits a unique solution $\psi\in \Sob^{1,2,p,\mu}((0, T)\times\Rd)\cap \Lp^{\infty}((0, T)\times\Rd)$, for some $p\ge 2$ and $\mu > 0$. Now, by It\^{o}-Krylov formula (as in \cite[Theorem~3.5.2]{HP09-book}), there exist an optimal Markov policy, i.e., there exists $v^*\in \Um$ such that $\cJ_{T}^{v^*}(x, c) = \cJ_{T}^*(x, c) = \psi(0,x)$\,.

Similarly, for each $n\in\NN$ (for the approximating models) the optimality equation
\begin{align}\label{EFinitecost1B}
&\frac{\partial \psi_n}{\partial t} + \inf_{\zeta\in \Act}\left[\sL_{\zeta}^n\psi_n + c_n(x, \zeta) \right] = 0 \\
& \psi_n(T,x) = H(x)
\end{align} admits a unique solution $\psi_n\in \Sob^{1,2,p,\mu}((0, T)\times\Rd)\cap \Lp^{\infty}((0, T)\times\Rd)$\,,\,\, $p\ge 2$\,. Moreover, by the It\^{o}-Krylov formula (as in \cite[Theorem~3.5.2]{HP09-book}), there exists $v_n^*\in \Um$ such that $\cJ_{T,n}^{v_n^*}(x, c_n) = \cJ_{T,n}^*(x, c_n) = \psi_n(0,x)$\,.

The following theorem shows that as the approximating model approaches the true model the optimal value of the approximating model converge to the optimal value of the true model. 
\begin{theorem}\label{FinitecostThm1}
Suppose Assumptions (A1), (A3) and (FN1) hold. Then 
\begin{equation*}
\lim_{n\to \infty} \cJ_{T,n}^*(x, c_n) = \cJ_{T}^*(x, c)\,.
\end{equation*}
\end{theorem}
\begin{proof} 
For any minimizing selector $v_n^*$ of \cref{EFinitecost1B}, we have  
\begin{align}\label{EFinitecost1C}
&\frac{\partial \psi_n}{\partial t} + \sL_{v_n^*}^n\psi_n + c_n(x, v_n^*(t,x)) = 0 \\
& \psi_n(T,x) = H(x)
\end{align}
By the It\^{o}-Krylov formula, it follows that
\begin{align}\label{EFinitecost1DA}
\psi_{n}(t,x) = \Exp_x^{v_n^*}\left[\int_t^{T} c_n(X_s^n, v_n^*(s, X_s^*)) \D{s} + H(X_T^*)\right]
\end{align}This implies that 
\begin{equation}\label{EFinitecost1D}
\norm{\psi_n}_{\infty} \leq T\norm{c_n}_{\infty} + \norm{H}_{\infty}\,.
\end{equation} Rewriting \cref{EFinitecost1C}, it follows that 
\begin{align*}
&\frac{\partial \psi_n}{\partial t} + \sL_{v_n^*}^n\psi_n + \lambda_0 \psi_n = \lambda_0 \psi_n - c_n(x, v_n^*(t,x))  \nonumber\\
& \psi_n(T,x) = H(x)\,,
\end{align*} for some fixed $\lambda_0 >0$\,. Thus, by parabolic pde estimate \cite[eq. (3.8), p. 234]{BL84-book}, we deduce that
\begin{equation}\label{EFinitecost1E}
\norm{\psi_n}_{\Sob^{1,2,p,\mu}} \leq \hat{\kappa}_1 \norm{\lambda_0 \psi_n - c_n(x, v_n^*(t,x))}_{\Lp^{p,\mu}}\,.
\end{equation} Thus, from \cref{EFinitecost1D} and \cref{EFinitecost1E}, we obtain $\norm{\psi_n}_{\Sob^{1,2,p,\mu}} \leq \hat{\kappa}_2$ for some positive constant $\hat{\kappa}_2$ (independent of $n$)\,. Since $\Sob^{1,2,p,\mu}((0, T)\times\Rd)$ is a reflexive Banach space, as a corollary of the Banach-Alaoglu theorem, there exists $\bar{\psi}\in\Sob^{1,2,p,\mu}((0, T)\times\Rd)$ such that along a subsequence (without loss of generality denoting by same sequence) 
\begin{equation}\label{EFinitecost1F}
\begin{cases}
\psi_n \to & \bar{\psi}\quad \text{in}\quad \Sob^{1,2,p,\mu}((0, T)\times\Rd)\quad\text{(weakly)}\\
\psi_n \to & \bar{\psi}\quad \text{in}\quad \Sob^{0,1,p,\mu}((0, T)\times\Rd)\quad \text{(strongly)}\,.
\end{cases}       
\end{equation} Now, as in our earlier analysis for the different cost criteria considered, multiplying both sides of the \cref{EFinitecost1B} by test function $\phi\in\cC_c^{\infty}((0, T)\times \Rd)$ and integrating, we get
\begin{align}\label{EFinitecost1G}
\int_{0}^{T}\int_{\Rd}\frac{\partial \psi_n}{\partial t}\phi(t,x)\D t \D x + \int_{0}^{T}\int_{\Rd}\inf_{\zeta\in \Act}\left[\sL_{\zeta}^n\psi_n + c_n(x, \zeta) \right]\phi(t,x)\D t \D x = 0\,.
\end{align}
Thus, in view of \cref{EFinitecost1F}, letting $n\to \infty$, from \cref{EFinitecost1G} it follows that (arguing as in \cref{ETC1.3C1A} - \cref{ETC1.3D})
\begin{align*}
\int_{0}^{T}\int_{\Rd}\frac{\partial \bar{\psi}}{\partial t}\phi(t,x)\D t \D x + \int_{0}^{T}\int_{\Rd}\inf_{\zeta\in \Act}\left[\sL_{\zeta}\bar{\psi} + c(x, \zeta) \right]\phi(t,x)\D t \D x = 0\,.
\end{align*}
Since $\phi\in\cC_c^{\infty}((0, T)\times \Rd)$ is arbitrary, from the above equation we deduce that $\bar{\psi}\in\Sob^{1,2,p,\mu}((0, T)\times\Rd)$ satisfies
\begin{align}\label{EFinitecost1H}
&\frac{\partial \bar{\psi}}{\partial t} + \inf_{\zeta\in \Act}\left[\sL_{\zeta}\bar{\psi} + c(x, \zeta) \right] = 0 \nonumber\\
& \bar{\psi}(T,x) = H(x)\,.
\end{align} Since $\psi $ is the unique solution of \cref{EFinitecost1H}, we deduce that $\bar{\psi}(0,x) = \psi(0,x) = \cJ_{T}^*(x, c)$. This completes the proof.
\end{proof}
In the following theorem, we prove the robustness result for the finite horizon cost criterion.
\begin{theorem}\label{FinitecostThm2}
Suppose Assumptions (A1), (A3) and (FN1) hold. Then for any optimal control $v_n^*$ of the approximating models we have
\begin{equation*}
\lim_{n\to \infty} \cJ_{T}^{v_n^*}(x, c) = \cJ_{T}^*(x, c)\,.
\end{equation*}
\end{theorem}
\begin{proof}
By the triangle inequality we have $$|\cJ_{T}^{v_n^*}(x, c) - \cJ_{T}^*(x, c)| \leq |\cJ_{T}^{v_n^*}(x, c) - \cJ_{T, n}^{v_n^*}(x, c_n)| + |\cJ_{T,n}^{v_n^*}(x, c_n) - \cJ_{T}^*(x, c)|\,.$$ From Theorem~\ref{FinitecostThm1}, it is known that $|\cJ_{T,n}^{v_n^*}(x, c_n) - \cJ_{T}^*(x, c)| \to 0$ as $n\to \infty$\,. Next, we show that $|\cJ_{T}^{v_n^*}(x, c) - \cJ_{T, n}^{v_n^*}(x, c_n)|\to 0$ as $n\to \infty$\,. 

Since the space $\Um$ is compact (with topology defined as in \cite[Definition~2.2]{YukselPradhan}), along a sub-sequence $v_n^*\to \bar{v}$. From \cite[Theorem~3.3, p. 235]{BL84-book}, we have that for each $n\in\NN$ there exists a unique solution $\bar{\psi}_n\in\Sob^{1,2,p,\mu}((0, T)\times\Rd)\cap \Lp^{\infty}((0, T)\times\Rd)$\,,\,\, $p\ge 2$, to the following Poisson equation
\begin{align}\label{EFinitecostThm2A}
&\frac{\partial \bar{\psi}_n}{\partial t} + \left[\sL_{v_n^*}\bar{\psi}_n + c(x, v_n^*(t,x)) \right] = 0 \nonumber\\
& \psi_n(T,x) = H(x)\,.
\end{align} By It\^{o}-Krylov formula, from \cref{EFinitecostThm2A} it follows that
\begin{align}\label{EFinitecostThm2B}
\bar{\psi}_{n}(t,x) = \Exp_x^{v_n^*}\left[\int_t^{T} c(X_s, v_n^*(s, X_s)) \D{s} + H(X_T)\right]
\end{align} 
This gives us 
\begin{equation}\label{EFinitecostThm2C}
\norm{\bar{\psi}_{n}}_{\infty} \leq T\norm{c}_{\infty} + \norm{H}_{\infty}\,.
\end{equation} Arguing as in Theorem~\ref{FinitecostThm1}, letting $n\to \infty$  from \cref{EFinitecostThm2A}, we deduce that there exists $\hat{\psi}\in\Sob^{1,2,p,\mu}((0, T)\times\Rd)\cap \Lp^{\infty}((0, T)\times\Rd)$\,,\,\, $p\ge 2$, satisfying
\begin{align}\label{EFinitecostThm2D}
&\frac{\partial \hat{\psi}}{\partial t} + \left[\sL_{\bar{v}}\hat{\psi} + c(x, \bar{v}(t,x)) \right] = 0 \nonumber\\
& \hat{\psi}(T,x) = H(x)\,.
\end{align}
Now using \cref{EFinitecostThm2D}, by It\^{o}-Krylov formula we deduce that
\begin{align}\label{EFinitecostThm2E}
\hat{\psi}(t,x) = \Exp_x^{\bar{v}}\left[\int_t^{T} c(X_s, \bar{v}(s, X_s)) \D{s} + H(X_T)\right]\,.
\end{align}

Moreover, we have 
\begin{align}\label{EFinitecostThm2F}
&\frac{\partial \psi_n}{\partial t} + \sL_{v_n^*}^n\psi_n + c_n(x, v_n^*(t,x)) = 0 \\
& \psi_n(T,x) = H(x)\,.
\end{align} Letting $n\to \infty$, as in Theorem~\ref{FinitecostThm1}, we have there exists $\tilde{\psi}\in\Sob^{1,2,p,\mu}((0, T)\times\Rd)\cap \Lp^{\infty}((0, T)\times\Rd)$\,,\,\, $p\ge 2$, satisfying
\begin{align}\label{EFinitecostThm2G}
&\frac{\partial \tilde{\psi}}{\partial t} + \left[\sL_{\bar{v}}\tilde{\psi} + c(x, \bar{v}(t,x)) \right] = 0 \nonumber\\
& \tilde{\psi}(T,x) = H(x)\,.
\end{align}
By It\^{o}-Krylov formula, from \cref{EFinitecostThm2G}, we obtain
\begin{align}\label{EFinitecostThm2H}
\tilde{\psi}(t,x) = \Exp_x^{\bar{v}}\left[\int_t^{T} c(X_s, \bar{v}(s, X_s)) \D{s} + H(X_T)\right]\,.
\end{align}  From \cref{EFinitecostThm2E}and \cref{EFinitecostThm2H}, we deduce that $\cJ_{T}^{v_n^*}(x, c) = \bar{\psi}_{n}(0,x)$ and $\cJ_{T, n}^{v_n^*}(x, c_n) = \psi_n(0,x)$ converge to the same limit\,. This completes the proof. 
\end{proof}

\section{Control up to an Exit Time}\label{exitTimeSection}
Before we conclude the paper, let us also briefly note that if one consider an optimal control up to an exit time with the cost given as:
\begin{itemize}
\item[•]\textit{(in true model:)} for each $U\in\Uadm$ the associated cost is given as
\begin{equation*}
\hat{\cJ}_{e}^{U}(x) \,\df \, \Exp_x^{U} \left[\int_0^{\tau(O)} e^{-\int_{0}^{t}\delta(X_s, U_s) \D s} c(X_t, U_t) \D t + e^{-\int_{0}^{\tau(O)}\delta(X_s, U_s) \D s}h(X_{\tau(O)})\right],\quad x\in\Rd\,,
\end{equation*}
\item[•]\textit{(in approximated models:)} for each $n\in \NN$ and $U\in\Uadm$ the associated cost is given as
\begin{equation*}
\hat{\cJ}_{e,n}^{U}(x) \,\df \, \Exp_x^{U} \left[\int_0^{\tau(O)} e^{-\int_{0}^{t}\delta(X_s, U_s) \D s} c_n(X_t, U_t) \D t + e^{-\int_{0}^{\tau(O)}\delta(X_s, U_s) \D s}h(X_{\tau(O)})\right],\quad x\in\Rd\,,
\end{equation*}
\end{itemize} where $O\subset \Rd$ is a smooth bounded domain, $\tau(O) \,\df\,  \inf\{t \geq 0: X_t\notin O\}$, $\delta(\cdot, \cdot): \bar{O}\times\Act\to [0, \infty)$ is the discount function and $h:\bar{O}\to \RR_+$ is the terminal cost function. In the true model the optimal value is defined as $\hat{\cJ}_{e}^{*}(x)=\inf_{U\in \Uadm}\hat{\cJ}_{e}^{U}(x)$, and in the approximated model the optimal value is defined as $\hat{\cJ}_{e,n}^{*}(x)=\inf_{U\in \Uadm}\hat{\cJ}_{e,n}^{U}(x)$\,. We assume that $\delta\in \cC(\bar{O}\times \Act)$, $h\in\cC(\bar{O})$. As in \cite{RZ21}, \cite[p.229]{B05Survey} the analysis leads to the following HJB equation. 
\begin{align*}
\min_{\zeta \in\Act}\left[\sL_{\zeta}\phi(x) - \delta(x, \zeta) \phi(x) + c(x, \zeta)\right] =  0\,,\quad \text{for all\ }\,\, x\in O\,,\quad\text{with}\quad
 \phi = h\,\,\, \text{on}\,\,\, \partial{O}\,.
\end{align*} By similar argument as in \cite[Theorem~3.5.3]{ABG-book}, \cite[Theorem~3.5.6]{ABG-book} we have that $\hat{\cJ}_{e}^{*}$, $\hat{\cJ}_{e,n}^{*}$ are unique solutions to their respective HJB equations. Existence follows by utilizing the Leray-Schauder fixed point theorem as in \cite[Theorem~3.5.3]{ABG-book} and uniqueness follows by It$\hat{o}$-Krylov formula as in \cite[Theorem~3.5.6]{ABG-book}\,. Using standard elliptic PDE estimates (on bounded domain $O$) and closely mimicking the arguments as in Theorem~\ref{TC1.3}, we have the following continuity result
\begin{theorem}\label{TExi1.1}
Suppose Assumptions (A1)-(A5) hold. Then 
\begin{equation*}
\lim_{n\to\infty} \hat{\cJ}_{e,n}^{*}(x) = \hat{\cJ}_{e}^{*}(x) \quad\text{for all}\,\, x\in \bar{O}\,.
\end{equation*}
\end{theorem}For each $n\in\NN$, suppose that $\hat{v}_{e,n}^*\in \Usm$, $\hat{v}_{e}^*\in \Usm$ are optimal controls of the approximated model and true model respectively. Then in view of the the above continuity result, following the steps of the proof of the Theorem~\ref{TC1.4}, we obtain the following robustness result.  
\begin{theorem}\label{TExi1.2}
Suppose Assumptions (A1)-(A5) hold. Then 
\begin{equation*}
\lim_{n\to\infty} \hat{\cJ}_{e}^{\hat{v}_{e,n}^*}(x) = \hat{\cJ}_{e}^{\hat{v}_{e}^*}(x) \quad\text{for all}\,\, x\in \bar{O}\,.
\end{equation*}
\end{theorem}
\section{Revisiting Example \ref{ERS11Example}}
Consider Example \ref{ERS11Example}(i). 
\begin{itemize}
\item[•]\textbf{For discounted cost:} Let $\hat{v}_n^*$ be a discounted cost optimal control when the system is governed by \cref{{ERS1.1}} (existence of such control is ensured by Theorem~\ref{TD1.1}). Then following Theorem~\ref{TC1.4}, we have that 
\begin{equation}\label{ENoiseAPPDisc}
\lim_{n\to\infty} \cJ_{\alpha}^{\hat{v}_n^*}(x, c) = \cJ_{\alpha}^{v^*}(x, c) \quad\text{for all}\,\, x\in\Rd\,.
\end{equation}
\item[•]\textbf{For ergodic cost:} Let $\hat{v}_n^*$ be an ergodic optimal control when the system is governed by \cref{{ERS1.1}} (existence is guaranteed by Theorem~\ref{ergodicnearmono2}, Theorem~\ref{TErgoOptApprox1}). Then arguing as in Theorem~\ref{ErgodNearmonoRobu1} (for near-monotone case) Theorem~\ref{ErgodLyapRobu1} (for stable case), it follows that 
\begin{equation}\label{ENoiseAPPErgo}
\lim_{n\to\infty} \inf_{x\in\Rd}\sE_{x}(c, v_{n}^*) = \sE^{*}(c)\,.
\end{equation}
\item[•]\textbf{Finite horizon cost:}
For each $n\in\NN$, let $\hat{v}_n^*$ be a finite horizon optimal control when the system is governed by \cref{{ERS1.1}}\,. Then in view of Theorem~\ref{FinitecostThm2}, we have 
\begin{equation}\label{EFiniteCost}
\lim_{n\to\infty} \cJ_{T}^{\hat{v}_{n}^*}(x, c) = \cJ_{T}^{*}(x, c) \quad\text{for all}\,\, x\in \Rd\,.
\end{equation}
\item[•]\textbf{For cost up to an exit time:} Let $\hat{v}_{e,n}^*$ be a an optimal control when the system is governed by \cref{{ERS1.1}}, for each $n\in \NN$. Then Theorem~\ref{TExi1.2} ensures that 
\begin{equation}\label{ENoiseExiCost}
\lim_{n\to\infty}\hat{\cJ}_{e}^{\hat{v}_{e,n}^*}(x) = \hat{\cJ}_{e}^{\hat{v}_{e}^*}(x) \quad\text{for all}\,\, x\in \bar{O}\,.
\end{equation}
\end{itemize}

\section{Conclusion}
In this paper, we studied continuity of optimal costs and robustness/stability of optimal control policies designed for an incorrect models applied to an actual model for both discounted/ergodic cost criteria. In our analysis we have crucially used the fact that our actual model is a non-degenerate diffusion model. It would be an interesting problem to investigate if such results can be proved in the cases when the limiting system (actual system) is a degenerate diffusion system. Also, in our analysis we have assumed that our system noise is given by a Wiener process; it would be interesting to study further noise processes e.g., when system noise is a wide-bandwidth process or a more general discontinuous martingale noise (as in \cite{K90}, \cite{KR87}, \cite{KR87a}, \cite{KR88})\,. In the latter case the controlled process may become non-Markovian process even under stationary Markov policies. Therefore, it is reasonable to find suitable Markovian approximation of it which maintains the necessary properties of the original system. The analysis of robustness problems in this setting is a direction of research worth pursuing. 


\bibliography{Somnath_Robustness}
\end{document}